\documentclass[12pt]{amsart}
\usepackage{xcolor}
\usepackage[draft]{todonotes}
\usepackage[english]{babel}
\usepackage{geometry}
\usepackage[T1]{fontenc}
\usepackage[latin1]{inputenc}
\usepackage{amsmath,amssymb,mathrsfs,amsthm, tikz-cd,mathrsfs}
\usepackage{soul}
\usepackage{epsfig}
\usepackage{hyperref}
\usepackage{float}
\usepackage{graphicx}
\usepackage{enumitem}
\usepackage{datetime}
\usepackage{xcolor}
\usepackage{fancyhdr}
\newtheorem{Theorem}{Theorem}[section]
\newtheorem{Proposition}[Theorem]{Proposition}

\newtheorem{lemma}[Theorem]{Lemma}

\newcommand{\fa}{\mathfrak{a}}

\newcommand{\fg}{\mathfrak{g}}
\newcommand{\fh}{\mathfrak{h}}

\newcommand {\ev} {{\bar0}}
\newcommand {\od} {{\bar1}}

\begin{document}

\title[Lagrangian extensions and  left-symmetric structures]{Lagrangian extensions and left-symmetric structures on the four-dimensional real Lie superalgebras}

\author{Sofiane Bouarroudj}
\address {Division of Science and Mathematics, New York University Abu Dhabi, P.O. Box 129188, Abu Dhabi, United Arab Emirates.}
\email{sb3922@nyu.edu; amr1190@nyu.edu}

\author{Ana-Maria Radu}
%\address {Division of Science and Mathematics, %New York University Abu Dhabi, P.O. Box 129188, %Abu Dhabi, United Arab Emirates.}
%\email{amr1190@nyu.edu}

% \thanks{SB is supported by the grant NYUAD-065.}

\keywords {Quasi-Frobenius Lie (super)algebra, Lagrangian extensions, Left-symmetric structures, Novikov structures, Balinsky-Novikov structures.
}
 \subjclass[2020]{17B05; 17B55; 17D25}

\begin{abstract} 
Over real numbers, Backhouse classified all four-dimensional Lie superalgebras. From this list, we  will investigate those Lie superalgebras that can be obtained as Lagrangian extensions. Moreover, we investigate left-symmetric structures on these Lie superalgebras. Furthermore, except for two of them, we show that they are all Novikov superalgebras. 

\end{abstract} 

\maketitle

\hspace{10pt}

\section{Introduction}

\subsection{Lagrangian extensions of Lie (super)algebras} The notion of {\it $T^*$-extension} of Lie algebras was initially introduced by Bordemann \cite{Bor}. A short explanation is as follows. Given a Lie algebra $\fh$ equipped with an $\fh^*$-valued bilinear map $\alpha$, a $T^*$-extension of $\fh$ is the Lie  algebra structure on the vector space  $\fg:=\fh\oplus \fh^*$ constructed by means of the form $\alpha$ that satisfies a certain cohomological requirement.  The Lie algebra $\fg$ naturally admits a non-degenerate invariant {\it symmetric} bilinear form.  Conversely, Bordemann showed in \cite{Bor} that every finite-dimensional nilpotent Lie algebra $\fh$ with an invariant symmetric non-degenerate bilinear form   ``is'' a suitable $T^*$-extension.

In particular,  if $\fh$ is a real finite-dimensional Lie algebra belonging to a Lie group $G$,  then the $T^*$-extension of $\fh$ is the Lie algebra of the cotangent bundle $T^*G$ of $G$,  considered as a Lie group.  The terminology and notation are justified by this differential geometric fact.

A superization of this construction was given in \cite{BBB}. In particular, it was shown that every solvable Lie superalgebra with an {\it even}  non-degenerate  \emph{symmetric} bilinear form is either isomorphic to a $T^*$-extension of a certain Lie superalgebra or to an ideal of codimension one of a certain $T^*$-extension.  

The notion of $T^*$-extension was also studied by Baues and Cort\'es in \cite{BC} in the context of Lie algebras admitting a~flat connection,  and called {\it Lagrangian extensions} since the Lie sub-algebra $\fh^*$ is a Lagrangian ideal in $\fh\oplus \fh^*$. Unlike the case studied in \cite{Bor}, the non-degenerate form in this context is \emph{anti-symmetric} and defines a 2-cocycle on $\fh$. In addition, the Lie algebra $\fh$ must admit a flat connection so that the Lagrangian extension is possible. Lie algebras with a non-degenerate anti-symmetric form satisfying the 2-cocycle conditions are called quasi-Frobenius, and are studied intensively in the literature, see, e.g., \cite{E, O}. 

Lie algebras possessing both a  non-degenerate invariant  symmetric bilinear form and a non-degenerate 2-cocycle of its scalar cohomology, were studied in \cite{BBM} and were called { \it quadratic symplectic} Lie algebras. If the ground field is algebraically closed, the authors showed that every quadratic symplectic Lie algebra is a $T^*$-extension of a Lie algebra which admits an invertible derivation. Necessary and sufficient conditions on $\fh$ and on the cocycle used in the construction of the $T^*$-extension are investigated to ensure that the extended algebra admits a 
skew-symmetric derivation and, hence, a symplectic structure. As a result, complex symplectic quadratic Lie algebras with dimensions less than or equal to 8 are classified. Some of these Lie algebras correspond to some Lie groups that admit a bi-invariant pseudo-Riemannian metric as well as a left-invariant symplectic form. These Lie groups are nilpotent, and their geometry is very rich, see \cite{BBM, MR} for more details.

The notion of Lagrangian extension was superized  in \cite{BM} (see also \cite{BBE} where the ground field is of characteristic 2). The construction is valid for arbitrary Lie superalgebras defined over a field of arbitrary characteristic  provided the Lie superalgebra admit a torsion-free flat connection. Moreover, they have observed that there are two ways to perform this extension: they  consider either $\fg:=\fh\oplus \fh^*$ or $\fg:=\fh\oplus \Pi (\fh^*)$,  where $\Pi$ is the change of parity functor.  They call them $T^*$-extension and $\Pi T^*$-extensions of $\fh$,  respectively. In the first case, the non-degenerate form is even, and in the second case, it is odd.  

Several examples of Lagrangian extensions are given in \cite{BM}, including filiform Lie superalgebras and some Lie superalgebras over reals. As a source of examples, they used the list of 4-dimensional real Lie superalgebras classified by Backhouse in \cite{B}.

One of the main objectives of this paper is to complete the list of examples among the 4-dimensional Lie superalgebras and explain how they can be obtained as Lagrangian extensions, see Section \ref{Tstar2}. According to the Backhouse classification, this investigation will classify all four-dimensional real Lie superalgebras that admit (or do not admit) Lagrangian extensions. 

\subsection{Left-symmetric structures}

Left-symmetric algebras (also known in the literature as Vinberg-Koszul algebras) play an important role in various fields of mathematics and mathematical physics; for instance, rooted tree algebras, operad theory, vertex algebras, convex homogeneous cones, affine manifolds, left-invariant affine structures on Lie groups, to name just a few. For a thorough review, see  \cite{B3}. 

 Over the real and complex numbers
left-symmetric algebras are closely related to the theory of left-invariant affine structures over Lie groups; see, for 
example, \cite{M, GS} (and references therein).

The investigation and classification of left-symmetric structure on Lie (super)algebras is an important subject on its own, as initiated in \cite{S}. The problem, however, seems to be difficult to handle.

For reductive Lie algebras, Baues \cite{Ba}  proved that $\mathfrak{gl}(n)$ is the only reductive Lie algebra with one-dimensional center and a simple semisimple ideal which admits left- symmetric algebras over an algebraically closed ground field (see also \cite{B1}). Moreover, all left-symmetric algebras on $\mathfrak{gl}(n)$ were classified. 

Simple Lie algebras over complex numbers do not admit left-symmetric structures, see~\cite{H}. The existence of such structures is not always guaranteed, even for nilpotent Lie algebras. There are filiform nilpotent Lie groups of dimension $10 \leq  n \leq 13$ that do not admit any left-invariant affine structure, see \cite{B, B3, BG}. Any filiform nilpotent Lie group of dimension $n \leq 9$ admits a left-invariant affine structure. 

Nevertheless, left-symmetric algebra structures on both finite-dimensional and infinite-dimensional Lie algebras have been studied extensively, see \cite{Bai, B2, YZ} as well as \cite{CL, KCB, TB}, respectively. Left-symmetric structures on the Virasoro superalgebra were studied in \cite{KB}.

In the super setting, it was proved in \cite{DZ} that  $\mathfrak{sl}(m +1|m)$ admits a left-symmetric superalgebra, and a full classification is offered in the case of  $\mathfrak{sl}(2|1)$. Moreover, it was shown that $\mathfrak{sl}(m|1)$ does not admit a left-symmetric structure for $m \geq 3$. 

Over a field of characteristic 2, all left-symmetric structures on a 2-dimensional superspace have been classified in \cite{BBE}.

There is also the notion of Novikov algebra. It was introduced in connection with the Poisson brackets of hydrodynamic type, see \cite{BN}.

A second aim of this paper is to provide explicitly a left-symmetric structure on the 4-dimensional real Lie superalgebras. Surprisingly  enough, all these Lie superalgebras admit such a structure, given explicitly in Section \ref{leftsym}. Furthermore, we show that these structures are Novikov, except for the two Lie superalgebras $(D_0^{10})^1$ and $(D_0^{10})^2$. In addition, all these Lie superalgebras admit a Balinsky-Novikov structure.

Classifying all left-symmetric structures on each of these Lie superalgebras remains an open problem.

\section{Main concepts and definitions}

The purpose of this section is to review the main definitions and theorems that will be used in this paper. 

Throughout the text,  $\mathbb K$ is a field of characteristic $p\not =2$. 

Let $V=V_{\bar 0}\oplus V_{\bar 1}$ be a superspace defined over  ${\mathbb K}$. The parity of a homogeneous element $v\in V_{\bar{i}}$ is denoted by $|v|:=\bar{i}$. The element $v$ is called \textit{even} if $v\in V_{\bar 0}$ and \textit{odd} if $v\in V_{\bar 1}$. The {\it superdimension} of $V$ is $\mathrm{sdim}(V)=a+b\epsilon$, where $\epsilon^2=1$, and $a=\mathrm{dim}(V_{\bar 0})$, $b=\mathrm{dim}(V_{\bar 1})$. Usually, $\mathrm{sdim}(V)$ is shorthanded as $a\mid b$. 

We will denote by $\Pi$ the functor that assigns to $V$ the superspace $\Pi(V)$ defined as the other copy of $V$ with the opposite parity of its homogeneous components. We have 
\[
(\Pi(V))_{\bar 0} = V_{\bar 1}, \quad  (\Pi(V))_{\bar 1} = V_{\bar 0}.
\]
The space $\Pi(V)$ consists of the linear combinations of elements that we denote by $\Pi(f)$ for every homogeneous $f\in V$.  For every superspace $V$,  denote the canonical odd homomorphism induced by the functor $\Pi$ by 
\[
\Pi: V \rightarrow \Pi(V) \quad f \mapsto \Pi(f).
\] 
We will identify $V$ and $\Pi(\Pi(V))$, naturally.

Let us introduce the notion of bilinear forms over a superspace. For more details, see \cite{LSoS}. 

Let $V$ and $W$ be two superspaces defined over an arbitrary field ${\mathbb K}$. Let $\mathcal{B} \in \text{Bil}(V,W)$ be a {\it homogeneous} bilinear form. The Gram matrix $B=(B_{ij})$ associated to ${\mathcal B}$ is given by the formula 
\begin{equation*}\label{martBil}
B_{ij}=(-1)^{|{\mathcal B}||v_i|}{\mathcal B} (v_{i}, w_{j})\text{~~for the basis vectors $v_{i}\in V$ and $w_{j}\in W$.}
\end{equation*}
This definition can be extended by linearity to non-homogeneous forms.  Moreover,  it allows us to identify a~bilinear form $B(V, W)$ with an element of $\mathrm{Hom}(V, W^*)$. 
Consider the \textit{upsetting} of bilinear forms
$u\colon\mathrm{Bil} (V, W)\rightarrow \mathrm{Bil}(W, V)$ given by the formula \begin{equation*}
\label{susyB}
u(\mathcal{B})(u, v)=(-1)^{|v||u|}\mathcal{B} (v,u)\text{~~for any homogeneous $v \in V$ and $u\in W$.}
\end{equation*}
From now on, let us assume that $V=W$. 

In terms of the Gram matrix $B$ of ${\mathcal B} $: the form
${\mathcal B}$ is  \textit{symmetric} if  and only if 
\begin{equation*}\label{BilSy}
u(B)=B,\;\text{ where $u(B)=
\left( \begin{array}{cc} R^{t} & (-1)^{|\omega |}T^{t}\\[2mm]
(-1)^{|\omega |}S^{t} & -U^{t}\end{array}\right ),$ for $B= \left (\begin{array}{cc} R & S\\[2mm]
T & U
\end{array} \right)$.}
\end{equation*}
Similarly, \textit{anti-symmetry} of ${\mathcal B}$ means that $u(B)=-B$.

Throughout the text, we denote by $\omega$ an anti-symmetric bilinear form. 

An even (resp. odd) non-degenerate anti-symmetric  bilinear form is called {\it orthosymplectic} (resp. {\it periplectic}).  

\begin{Proposition}[Superdimension constraints]
If $V=V_{\bar{0}} \oplus V_{\bar{1}}$ is a finite-dimensional superspace such that $V_{\bar{1}} \not = {0}$, equipped with a non-degenerate anti-symmetric bilinear form $\omega$, then:

\begin{enumerate}[label=(\roman*)]

    \item[\textup{(}i\textup{)}] If $|\omega|=\bar{0}$, then $\mathrm{dim} (V_{\bar{0}})$ is even.\\
    \item[\textup{(}ii\textup{)}] If $|\omega|=\bar{1}$, then $\mathrm{dim}(V_{\bar{0}})=\mathrm{dim} (V_{\bar{1}})$.
    
\end{enumerate}

\end{Proposition}
\begin{proof} See, \cite{BM} \end{proof}
A Lie superalgebra is a superspace $\mathfrak{g} =  \mathfrak{g}_\ev \oplus \mathfrak{g}_\od$ over a field of characteristic $p \not = 2$ together with a bilinear binary operator $[\cdot, \cdot] : \mathfrak{g} \times \mathfrak{g} \rightarrow \mathfrak{g}$, called the Lie bracket that has the following properties: (for all homogeneous $x, y, z \in \mathfrak{g}$)

\begin{enumerate}
    \item[(i)] Super anti-commutativity: $[x, y] = -(-1)^{|x||y|}[y, x]$;
    \item[(ii)] Super Jacobi identity: $[x, [y, z]] = [[x, y], z] + (-1)^{|x||y|}[y, [x, z]]$ where $|x|={\bar 0}$ if $x \in \mathfrak{g}_{\bar 0}$ and $|x|={\bar 1}$ if $x \in \mathfrak{g}_{\bar 1}$;
    \item [(iii)] If $p=3$, one must add $[x,[x,x]]=0$ if $x\in \fg_{\bar 1}$.

\end{enumerate}
There is also the notion of Lie superalgebra in characteristic $2$ but we are not studying them here.

Let $\omega \in \mathrm{Bil}(\mathfrak{g}, \mathfrak{g})$ be an anti-symmetric bilinear form on a Lie superalgebra $\mathfrak g$. The form $\omega$ is called \emph{closed} if it satisfies
\[(-1)^{|x||z|}\omega(x, [y,z])+(-1)^{|z||y|}\omega(z,[x,y])+(-1)^{|y||x|}\omega(y,[z,x])=0,\]
for all homogeneous $x,y,z\in \fg_{\bar i}$. This means that $\omega$ is 2-cocycle with scalar values. 

A Lie superalgebra $\fg$ is called \textit{quasi-Frobenius} if it is equipped with a closed non-degenerate anti-symmetric bilinear form $\omega$. If $\omega$ is even (resp. odd) we call it {\it orthosymplectic} (resp. {\it periplectic}). We denote such a Lie superalgebra by $(\fg, \omega)$. Recently, these Lie superalgebras were studied in \cite{BBE, BE, BEM, BM} in the context of Lagrangian and double extensions.  

A Lie superalgebra  is called \textit{Frobenius} if $\omega$ is exact, which means that there is $f\in \mathfrak{g}^*$ such that $\omega(x,y)=f([x,y])$ for all $x,y \in \mathfrak{g}.$

Let $(\fg,\omega)$ be a quasi-Frobenius Lie superalgebra. An ideal $I\subseteq \fg$ is called {\it Lagrangian} if and only if $I^\perp=I$. 

%We will be using the following Lemma.
%\begin{lemma} [The center of $\mathfrak{g}$]Let %$\fg$ be a Lie superalgebra with  with a~non-%degenerate and antisymmetric form $\omega$. %Then ${\mathfrak z}(\fg)\subseteq %([\fg,\fg])^\perp$.
%\end{lemma}
%\begin{proof}
%See, \cite{BM}.
%\end{proof}
\subsection{Connections on Lie (super)algebras}
To write down a Lagrangian extension of a Lie (super)algebra, we will need the notion of a connection on a Lie (super)algebra. In our context, connections are defined as non-associative products on a superalgebra with purely algebraic definitions. See \cite{BM,MKSV} for the non-super case. 

An \emph{even} bilinear map $\nabla: \mathfrak{g} \times \mathfrak{g} \rightarrow \mathfrak{g}$, written as $(x,y) \rightarrow \nabla_x y $ is called a \emph{connection} on $ \mathfrak{g}.$

We will introduce the notion of  \emph{torsion} and \emph{curvature} of a connection. 

Let $\nabla : \mathfrak{g} \times \mathfrak{g} \rightarrow \mathfrak{g}$ be a connection on $\mathfrak g$. For any homogeneous  $x,y,z \in \mathfrak{g}_{\bar i}$, the  \textit{torsion} $T^{\nabla}$ of $\nabla$ is defined as 
\begin{equation}\label{TorCur} 
T(x,y) := \nabla_xy -(-1)^{|x||y|}\nabla_y x - [x,y].
\end{equation}
The \textit{curvature} $R^{\nabla}$ of $\nabla$ is defined as 
\begin{equation}\label{Cur} 
R^\nabla(x,y)h := \nabla _x \nabla_y z-(-1)^{|x||y|)}\nabla_y \nabla_x z - \nabla_{[x,y]}z.
\end{equation}
By bilinearity, the definition of  $T$ and $R^\nabla$ can be extended to the whole $\mathfrak g$. 

If $T=0$, then the connection $\nabla$ is called {\it torsion-free}. The curvature $R^\nabla$ vanishes if and only if the map $\sigma^\nabla:  \mathfrak{g} \rightarrow \mathrm{End}(\mathfrak{g}), \; x \mapsto \nabla_x$ is a representation of $\mathfrak{g}$ on itself. In this case, the connection $\nabla$ will be called {\it flat}. 

\subsection{Left-symmetric (super)algebras}
Let us review a few notions regarding left- symmetric structures. 

A superalgebra $L =L_{\bar 0} \oplus L_{\bar 1}$ over a field $\mathbb K$ equipped with an even  bilinear operation $\cdot$ is called a left-symmetric superalgebra (or an LSSA for short) if the associator 
\[
(x, y, z) :=(x \cdot y) \cdot z- x  \cdot (y \cdot z)
\]
is supersymmetric in $x$ and $y$, i.e., $(x, y, z) =(-1)^{|x||y|}(y, x, z)$; or, equivalently,
\[
(x \cdot y) \cdot z - x \cdot (y \cdot z) = (-1)^{|x||y|}((y \cdot x) \cdot z-y \cdot (x \cdot z)), \text{ for all } x, y,z \in L_{\bar i}. 
\]
These superalgebras are Lie-admissible. The supercommutator $[x, y] :=x \cdot y-(-1)^{|x||y|} y \cdot x$ defines then a Lie superalgebra structure on $L$ and we denote it by ${\mathfrak g}(L)$. The resulting Lie superalgebra is called the associated Lie superalgebra of the LSSA. Note that if $L$ admits an LSSA then the even part $L_{\bar 0}$ also admits one. Associative superalgebras are all LSSAs. 

In \cite{H}, Helmstetter showed  that it is necessary to have $[L,L] \subsetneq L$ if $L$ is a left-symmetric algebra defined over the field of complex numbers. Over a field of characteristic $p$, this is not true anymore. Burde showed in \cite{B2} that classical Lie algebras over a field $p>3$ admit an LSA only in case $p$ divides $\mathrm{dim}(L)$.  In the super setting and over a field of prime characteristic, there has not been much research done so far. 

The existence of a flat connection $\nabla$ on a Lie superalgebra $\mathfrak g$ will induce a left-symmetric structure given by 
\[
x\cdot y:=\nabla_x(y) \text{ for all }x,y \in {\mathfrak g}.
\]
Additionally, every quasi-Frobenius Lie superalgebra $({\mathfrak g}, \omega) $ admits a left-symmetric structure given by
\[
\omega(x\cdot y, z)=(-1)^{|x||y|}\omega(y, [x,z]) \text{ for all }x,y \in {\mathfrak g}_{\bar i}, \text{ and } z \in {\mathfrak g}.
\]
The converse is not always true. Section~\ref{leftsym} provides several examples illustrating this. 

Certain subclasses of LSSAs are also of particular interest.  An LSSA with supercommutative right multiplication is called a {\it Novikov superalgebra} (cf. \cite{BN, Xu}). Explicitly, a {\it Novikov}  superalgebra $(L, \cdot)$ is an LSSA with an additional condition
\[
(z \cdot x) \cdot y =(-1)^{|x||y|}(z \cdot y) \cdot x,\;  \text{ for all } x,y \in L_{\bar i}, \text{ and } z \in  L.
\]
It turns out that there are two superizations of the notion of Novikov algebras. The second superization is due to \cite{Bal} (see also \cite{PB}) and will be referred to as {\it Balinsky Novikov superalgebras}, or {\it BN superalgebras} for short. They are defined as follows: 
\begin{enumerate}
    \item Left-symmetry and commutativity:
    \[
    (x \cdot y)\cdot z-x \cdot (y \cdot z)= (y \cdot x) \cdot z- y \cdot (x \cdot z), \qquad (x \cdot y) \cdot z= (x \cdot z)\cdot y,
    \]
for all homogeneous elements
\[
x,y,z \in L_{\bar 0}; \quad x,y\in L_{\bar 0}, z \in L_{\bar 1};  \quad x,z\in L_{\bar 0}, y \in L_{\bar 1};  \quad y,z\in L_{\bar 0}, x \in L_{\bar 1}; 
\]
    \item Commutativity: $x\cdot y=y \cdot x$ for all $x,y\in L_{\bar 1}$;

    \item Compatibility conditions:
    \[
    \begin{array}{ll}
    (x \cdot y)\cdot z= (x \cdot z)\cdot y & \text{ for all } x,y\in L_{\bar 1} \text{ and } z\in  L_{\bar 0};\\[2mm]
    x \cdot (y \cdot z)=(x \cdot y)\cdot z+ y \cdot (x \cdot z) & \text{ for all } x\in L_{\bar 0} \text{ and } y, z\in  L_{\bar 1};\\[2mm]
    x \cdot (y \cdot z)= (x \cdot y)\cdot z + (x \cdot z)\cdot y & \text{ for all } x,y, z\in L_{\bar 1}.\\[2mm]
    \end{array}
    \]
\end{enumerate}
Balinsky Novikov superalgebras are also admissible Lie superalgebras. The Lie structure is given as follows: 
\[
\begin{array}{ll}
[x,y]:=x \cdot y - y \cdot x, & \text{ for all } x,y\in L_{\bar 0};\\[2mm]

[x,y]:=x \cdot y - \frac{1}{2} y \cdot x=-[y,x], & \text{ for all } x\in L_{\bar 0}, \text{ and } y\in L_{\bar 1}\\[2mm]

[x,y]:=x \cdot y  & \text{ for all } x,y\in L_{\bar 1}.
\end{array}
\]
%%%%%%%%%%%%%%%%%%%%%%%%%%%%%%%%%%%%%%%%%%%%%%%
\subsection{The 4-dimensional real Lie superalgebras}
The ground field $\mathbb K$ is assumed to be the field of real numbers $\mathbb R$. The classification of four-dimensional real Lie superalgebras has been carried out by Backhouse, see \cite{B}. The Lie superalgebras with anti-symmetric bilinear forms were examined in \cite{BM}. %They also deduced the following result superizing a result in \cite{BYC}. 

%\begin{Theorem}[Solvability of %$\mathfrak{g}$]
%Every 4-dimensional real Lie %superalgebra equipped with an non-%degenerate anti-symmetric bilinear %form is solvable.
%
%\end{Theorem}

As in the \cite{BM} paper, the tables in the Appendix represent the classification of 4-dimensional real Lie superalgebras as follows:
\begin{enumerate}
    \item \textbf{Table 1} having the {\it trivial algebras}, i.e.  $[\mathfrak{g_{\bar{1}}},\mathfrak{g_{\bar{1}}}]=\{0\}$ and with  $\mathrm{sdim} = 2|2$;
    \item \textbf{Table 2} having the {\it non-trivial algebras}, i.e. $[\mathfrak{g_{\bar{1}}},\mathfrak{g_{\bar{1}}}]\not = \{0\}$ and  with $\mathrm{sdim} = 2|2$;
    \item \textbf{Table 3} having the {\it trivial algebras}, i.e.  $[\mathfrak{g_{\bar{1}}},\mathfrak{g_{\bar{1}}}]=\{0\}$ and with $\mathrm{sdim} = 3|1$; 
    \item \textbf{Table 4} having the {\it trivial algebras}, i.e.  $[\mathfrak{g_{\bar{1}}},\mathfrak{g_{\bar{1}}}]= \{0\}$ and with $\mathrm{sdim} = 1|3$; 
    \item \textbf{Table 5} having the {\it non-trivial algebras}, i.e. $[\mathfrak{g_{\bar{1}}},\mathfrak{g_{\bar{1}}}] \not = \{0\}$ and  with $\mathrm{sdim} = 3|1$; 
    \item \textbf{Table 6} having the {\it non-trivial algebras}, i.e. $[\mathfrak{g_{\bar{1}}},\mathfrak{g_{\bar{1}}}]\not = \{0\}$ and with $\mathrm{sdim} = 1|3$.
\end{enumerate}
In this paper, we present only the forms as in Table 1-6 which are suitable for Lagrangian extensions, as opposed to \cite{BM}, where the forms are given in the form of parametric families. The classification of these structures up to a symplectomorphism is still an open problem. We also correct a claim in \cite{BM} stating that only non-homogeneous forms exist for the Lie superalgebras $(D_0^{10})^1$ and $(D_0^{10})^2$. It turns out that these two superalgebras both admit an even and an odd non-degenerate closed form, and both of them can be obtained as Lagrangian extensions. We only consider \emph{indecomposable} Lie superalgebras.

\section {The notion of \texorpdfstring{$T^*$-}{T*-} extensions and \texorpdfstring{$\Pi T^*$-}{Pi T*-} extensions} We will introduce the notion of $T^*$-extension and $\Pi T^*$-extension of Lie superalgebras (Lagrangian extensions). The notion of a $T^*$-extension for Lie algebras was originally introduced by Bordemann in \cite{Bor}.  Recently, it was studied in \cite{BC} in the context of quasi-Frobenius Lie algebras, and then in \cite{BBE, BM} in the context of quasi-Frobenius Lie superalgebras. 
\subsection{Flat connections and representations}
Let $({\mathfrak h}, \nabla)$ be a Lie superalgebra endowed with a flat connection $\nabla$. Since the connection $\nabla$ is flat, it follows that the map 
\[
\mathfrak{h} \rightarrow \mathrm{End}(\mathfrak{h}) \quad u\mapsto \nabla_u,
\]
defines a representation. We define a map $\rho: \mathfrak{h} \rightarrow \mathrm{End}(\mathfrak{h}^*)$, the dual representation, as follows: (for all homogeneous $u\in \mathfrak{h}$ and $\xi\in \mathfrak{h}^*$)
\begin{equation}\label{eq5.1}
\rho: \fh \rightarrow \mathrm{End}(\mathfrak{h}^*) \quad u \mapsto \rho(u), \text{ where } \rho(u)\cdot \xi =- (-1)^{|u| |\xi|} \xi \circ \nabla_u.
\end{equation} 
Similarly, we have a representation 
\begin{equation}\label{eq5.1b}
\chi: \fh \rightarrow \mathrm{End}(\Pi(\mathfrak{h}^*)) \quad u\mapsto \chi(u):= (-1)^{|u|}\, \Pi \circ \rho(u)\circ \Pi.
\end{equation} 

\begin{lemma}[Dual representation]
The maps \eqref{eq5.1} and \eqref{eq5.1b} are indeed representations. 

\end{lemma}
\begin{proof}
See \cite{BM}. 
\end{proof}
These two representations will be used to build Lagrangian extensions in the next subsection. 
%%%%%%%%%%%%%%%%%%%%%%%%%%%%%%%%%%%%%%
\subsection{Polarization and Lagrangian extensions}\label{Lagrangian}
%%%%%%%%%%%%%%%%%%%%%%%%%%%%%%%%%%%%%%
Following \cite{BC} (for the supercase, see \cite{BBE, BM}), a {\it polarization} for a~quasi-Frobenius Lie superalgebra $(\mathfrak{g}, \omega)$ is a choice of a homogeneous Lagrangian subalgebra $\mathfrak{l}$ of $(\fg, \omega)$ (namely, $\mathfrak{l}=\mathfrak{l}^{\perp}$, with $\mathfrak{l}^{\perp}$ the orthogonal with respect to $\omega$). A {\it strong polarization} of a~quasi-Frobenius Lie superalgebra $(\mathfrak{g}, \omega)$ is a pair $(\mathfrak{a}, N)$ consisting of a homogeneous Lagrangian {\it ideal} $\fa\subset \fg$ and a complementary Lagrangian {\it subspace}  $N\subset \fg$. Following \cite{BC}, the quadruple $(\fg, \omega,\fa, N)$ is referred to as a {\it strongly polarized}  quasi-Frobenius Lie superalgebra.

The following construction is borrowed from \cite{BM}, as a generalization to \cite{BC}.

By means of a 2-cocycle $\alpha \in Z^2(\mathfrak{h}, \mathfrak{h}^*)$ (resp. $\beta  \in Z^2(\mathfrak{h}, \Pi(\mathfrak{h}^*))$ we will construct a Lie superalgebra structure on $\mathfrak{g}:=\mathfrak{h}\oplus \mathfrak{h}^*$ (resp. $\mathfrak{g}:=\mathfrak{h}\oplus \Pi(\mathfrak{h}^*)$), called $T^*$-extension (resp. $\Pi T^*$-extension) (Lagrangian extension). By construction, we will see that these Lie superalgebras are strongly polarized quasi-Frobenius.

Recall that the two spaces $\mathfrak{h}^*$ and $\Pi(\mathfrak{h}^*)$ are $\mathfrak{h}$-modules by means of the representations  (\ref{eq5.1}) and (\ref{eq5.1b}).\\

\underline{On the space $\mathfrak{g}:=\mathfrak{h}\oplus \mathfrak{h}^*$:} The brackets are defined as follows: \textup{(}for all $u, v\in \mathfrak{h}$ and $\xi\in \mathfrak{h}^*$\textup{)}
\[
\begin{array}{l}
 [u,v]_\mathfrak{g}  :=  [u,v]_{\mathfrak{h}} +  \alpha(u,v), \quad [u,\xi]_{\mathfrak{g}} :=  \rho(u) \cdot \xi.
\end{array}
\]
An \emph{even} form $\omega$ is defined on $\mathfrak{g}$ as follows
\begin{equation}\label{omegaTstar}
\omega(u+\xi, v+ \zeta):= \xi(v) - (-1)^{|\zeta|| u|} \zeta(u),
\end{equation}
for all homogeneous $u+\xi, v+\zeta \in \mathfrak{h} \oplus \mathfrak{h}^*$.

This construction will be referred to as the {\it $T^*$-extension}.% and denoted by $(\mathfrak{h}\oplus \mathfrak{h}^*,\alpha, \nabla, \omega)$. 

\underline{On the space $\mathfrak{k}:=\mathfrak{h}\oplus \Pi(\mathfrak{h}^*)$:} The brackets are defined as follows: \textup{(}for any $u, v\in \mathfrak{a}$\textup{)}
\[
\begin{array}{l}
 [u,v]_\mathfrak{g}  :=  [u,v]_\mathfrak{h} +  \beta(u,v), \quad [u, \Pi(\xi)]_\mathfrak{k} :=  \chi (u) \cdot \Pi(\xi).
\end{array}
\]
We define an \emph{odd} form $\kappa$ on $\mathfrak{k}$ as follows
\begin{equation}\label{omegaTstarodd}
\kappa(u+\Pi(\xi), v+ \Pi(\zeta)):= \xi(v) - (-1)^{(|\zeta|+1) |u|} \zeta (u),
\end{equation}
for all homogeneous $u+\Pi(\xi), v+\Pi(\zeta) \in \mathfrak{h}\oplus \Pi(\mathfrak{h}^*)$.

This construction will be referred to as the {\it $\Pi T^*$-extension}.% and denoted by $(\mathfrak{h}\oplus \Pi(\mathfrak{h}^*), \beta, \nabla, \kappa)$.  

\begin{lemma}[Conditions on $\alpha$ and $\beta$] \label{condalphabeta}\textup{(}i\textup{)} The form $\omega$ is closed on $\mathfrak{g}$ if and only if 
\begin{equation}\label{eq5.4}
(-1)^{|u||w|}\alpha(u,v)(w)+ \circlearrowleft (u,v,w)=0 \text{ for all homogeneous  $u,v,w\in\mathfrak{h}$},
\end{equation}
where here and in what follows, the symbol $\circlearrowleft(u,v,w)$ denotes the sum over cyclic permutations 
of the variables $u,v,w$.

\textup{(}ii\textup{)} The form $\kappa$ is closed on $\mathfrak{g}$ if and only if 
\begin{equation}\label{eq5.4b}
(-1)^{|u||w|}\Pi \circ \beta(u,v)(w)+ \circlearrowleft (u,v,w)=0 \text{ for all homogeneous $u,v,w\in\mathfrak{h}$}.
\end{equation}

\end{lemma}
\begin{proof}
See, \cite{BM}.
\end{proof}
We arrive to the following theorem proved in \cite{BM} (see also \cite{BC} in the non-super case).

\begin{Theorem} [Lagrangian or $T^*$- and $\Pi T^*$-extensions] \label{Tstar}
Let $(\fh, \nabla)$ be a flat Lie superalgebra. To every 2-cocycle $\alpha\in Z^2(\mathfrak{h}, \mathfrak{h}^*)$ \textup{(}resp. $\beta\in Z^2(\mathfrak{h}, \Pi(\mathfrak{h}^*))$\textup{)} which satisfies \eqref{eq5.4}  \textup{(}resp. \eqref{eq5.4b}\textup{)} one can canonically associate a strongly polarized quasi-Frobenius Lie superalgebra $\textup{(}\mathfrak{g}, \omega, \mathfrak{h}^*, \mathfrak{h})$ \textup{(}resp. $\textup{(}\mathfrak{k}, \kappa,\Pi(\mathfrak{h}^*), \mathfrak{h}$\textup{)}, where $\mathfrak{g}=\mathfrak{h}\oplus \mathfrak{h}^*$ \textup{(}resp. $\mathfrak{k}=\mathfrak{h}\oplus \Pi(\mathfrak{h}^*)$\textup{)} and $\omega$ \textup{(}resp. $\kappa$\textup{)} is given as in \eqref{omegaTstar} \textup{(}resp.~\eqref{omegaTstarodd}\textup{)}.
\end{Theorem}

The converse of this theorem says that every orthosymplectic or periplectic quasi-Frobenius Lie superalgebra $(\mathfrak{g}, \omega)$ that has a homogeneous Lagrangian ideal $\mathfrak{a}$ can be obtained as a $T^*$-extension  or $\Pi T^*$-extension of a flat Lie superalgebra. 

Let $(\mathfrak{g}, \omega, \mathfrak{a},N)$ be a strongly polarized quasi-Frobenius Lie superalgebra. Here the form $\omega$ is either orthosymplectic or periplectic. Consider the quotient space 
$\mathfrak{h}=\mathfrak{g}/\mathfrak{a}$ (recall that $\fa^\perp=\fa$). We will construct a flat connection $\nabla$ on $\mathfrak{h}$ (see \cite{BM}), superizing the construction of \cite{BC}. For every homogeneous $u,v\in \fh$, we  denote by $\tilde u$ and $\tilde v$ their lift to $\fg$. We then write 
\begin{equation}\label{eq5.5}
\omega_\mathfrak{h}(\nabla_uv, a)= - (-1)^{|u||v|} \omega_\mathfrak{g} (\tilde u, [\tilde v, a]).
\end{equation}
The pair $(\mathfrak{h},\nabla)$ is called {\it the quotient flat Lie superalgebra associated with the Lagrangian ideal} $\mathfrak{a}$ of $(\mathfrak{g}, \omega)$.

The following theorem is needed to determine whether the Lie superalgebras in Backhouse's list arise as Lagrangian extensions of smaller Lie superalgebras. We refer to \cite{BM} for its proof.

\begin{Theorem} [Converse of Theorem \ref{Tstar}] \label{TstarConv} Let $(\mathfrak{g}, \omega, \mathfrak{a},N)$ be a strongly polarized ortho-symplectic or periplectic quasi-Frobenius Lie superalgebra and $(\mathfrak{h}, \nabla)$ be its associated quotient flat Lie superalgebra. 

\textup{(}i\textup{)} If the form $\omega$ is orthosymplectic, then there exists $\alpha \in Z^2(\mathfrak{h}, \mathfrak{h}^*)$ 
satisfying Eq. \eqref{eq5.4}, such that $(\mathfrak{g}, \omega, \mathfrak{a},N)$ is isomorphic to the $T^*$-extension of $(\mathfrak{h},\nabla)$ by means of $\alpha$.

\textup{(}ii\textup{)} If the form $\omega$ is periplectic, then there exists $\beta \in Z^2(\mathfrak{h}, \Pi(\mathfrak{h}^*))$ 
satisfying Eq. \eqref{eq5.4b}, such that $(\mathfrak{g}, \omega, \mathfrak{a},N)$ is isomorphic to the $\Pi T^*$-extension of $(\mathfrak{h},\nabla)$ by means of $\beta$.

\end{Theorem}
% \begin{proof}
% See \cite{BM}.
% \end{proof}
\section{The 4-dimensional Lie superalgebras as Lagrangian extensions} \label{Tstar2}
Using Theorem \ref{Tstar} and Theorem  \ref{TstarConv} above,  we will determine whether the 4-dimensional real Lie superalgebras classified by Backhouse (see the Appendix) are isomorphic to a $T^*$-extension or to a $\Pi T^*$-extension of a smaller Lie superalgebra. 

We will outline the proof and computation for the Lie superalgebra $D^6$ (row 2 in the table); however, we will only exhibit the connection $\nabla$ for the other Lie superalgebras. 

\begin{enumerate}
    \item \underline{The Lie superalgebra $D^5$}: We correct here the description of the  Lie superalgebra $D_5$ in \cite{BM}. It is  $\Pi T^*$-extension of the Lie superalgebra $\mathfrak{h}=\mathrm{Span}\{e_1 | e_4\}$ with the bracket $[e_1,e_4]_{\mathfrak{h}} = e_4$, and where $\Pi \mathfrak{h}^*$ is spanned by $e_2 = -\Pi(e_4^*) | e_3 = \Pi(e_1^*)$. The 2-form $\beta \equiv 0$ and the connection $\nabla$ on $\mathfrak{h}$ is given by:
\[\nabla_{e_4}(e_1) = -e_4,\quad  \nabla_{e_1}(e_1) = -e_1, \quad  \nabla_{e_1}(e_4) = 0, \quad \nabla_{e_4}(e_4) = 0.\]

    \item \underline{The Lie superalgebra $D^6$:} We will  show first that this Lie superalgebra does have a Lagrangian ideal. Indeed, consider $\mathfrak{a}=\mathrm{Span}\{e_3, e_4 \} $. This is obviously an ideal because of the multiplication table. Let us show it is Lagrangian. We have
    \[\mathfrak{a}^{\bot}=\{v\in D^6 \;|\; \omega(v, x)=0, \forall x\in \mathfrak{a} \}, \]
    where the form we are considering is given by  $\omega=e_3^*\wedge e_1^* + e_4^*\wedge e_2^* $. 
    Let $v=l_1e_1+l_2e_2+l_3e_3+l_4e_4$. We have 
    \[ 0 = \omega(v,e_3)= -l_1 \implies l_1=0,\]
    \[ 0 = \omega(v,e_4)= -l_2 \implies l_2=0.\]
    This implies that $\mathfrak{a}^{\bot}=\mathfrak{a}$. The Lie superalgebra $\mathfrak h$ is given by $\mathrm{Span}\{e_1 , e_2\}$ with the bracket $[e_1,e_2]_{\mathfrak{h}} = 0$. The connection $\nabla$ is given by
     \[
     \nabla_{e_1}(e_1) = -e_1,\quad  \nabla_{e_1}(e_2) = -e_2, \quad  \nabla_{e_2}(e_1) = -e_2, \quad \nabla_{e_2}(e_2) = e_1.
     \]

    Now, we compute the torsion and curvature of $\nabla$, using Eqns. (\ref{TorCur}, \ref{Cur}). For the torsion, we have
    \[ 
T(e_1,e_2) = \nabla_{e_1}(e_2) -(-1)^{|e_1||e_2|}\nabla_{e_2} (e_1) - [e_1,e_2]_{\mathfrak h} =
    -e_2-(-e_2)-0=0.
    \]
    For the curvature, we have   
    \[ 
    \begin{array}{lcl} R^\nabla(e_1,e_2)(e_1) & = &  \nabla _{e_1} \nabla_{e_2} (e_1)-(-1)^{|e_1||e_2|}\nabla_{e_2} \nabla_{e_1} (e_1) - \nabla_{[e_1,e_2]_{\mathfrak h}}(e_1)\\[2mm]
    &= & \nabla_{e_1}(-e_2)-\nabla_{e_2}(-e_1)=e_2-e_2=0.
    \end{array}
    \]
    Similarly, 
    \[ \begin{array}{lcl}
    R^\nabla(e_1,e_2)(e_2) & = & \nabla _{e_1} \nabla_{e_2} (e_2)-(-1)^{|e_1||e_2|}\nabla_{e_2} \nabla_{e_1} (e_2) - \nabla_{[e_1,e_2]_{\mathfrak h}}(e_2)\\[2mm]
    & = & 
    \nabla_{e_1}(e_1)-\nabla_{e_2}(-e_2)=-e_1+e_1=0.
    \end{array}
    \]
  Therefore, the connection is torsion-free and flat.
    
Thus, the Lie superalgebra $D^6$ is a $\Pi T^*$-extension of the Lie superalgebra $\mathfrak{h}$,  where $\Pi \mathfrak{h}^*$ is spanned by $e_3 = \Pi(e_1^*) | e_4 = \Pi(e_2^*)$.  The 2-form $\beta \equiv 0$.

    \item \underline{The Lie superalgebra $D^{7}_{pq}$:} This Lie superalgebra is not quasi-Frobenius.\\

    \item \underline{The Lie superalgebra $D^{7}_{-1q}$, with $q\le1$:} This Lie superalgebra is a $\Pi T^*$-extension of the Lie superalgebra $\mathfrak{h}=\mathrm{Span}\{e_1 , e_2\}$ with the bracket $[e_1,e_2]_{\fh} = e_2$, and where $\Pi \mathfrak{h}^*$ is spanned by $e_4 = \Pi(e_1^*) , e_3 = \Pi(e_2^*)$. The 2-form $\beta \equiv 0$ and the connection $\nabla$ is given by
    
    \[\nabla_{e_1}(e_1) = -qe_1,\quad  \nabla_{e_1}(e_2) = e_2, \quad  \nabla_{e_2}(e_1) = 0, \quad \nabla_{e_2}(e_2) = 0.\]

%    \item \underline{The Lie superalgebra $D^{7}_{pp}$, with $p=-1$:} This Lie superalgebra is a $\Pi T^*$-extension of the Lie superalgebra $\mathfrak{h}=\mathrm{Span}\{e_1 , e_2\}$ with the bracket $[e_1,e_2]_\fh = e_2$, and where $\Pi \mathfrak{h}^*$ is spanned by $e_3 = \Pi(e_1^{*}) , e_4 = \Pi(e_2^{*})$. The 2-form $\beta \equiv 0$ and the connection $\nabla$ is given by:
    
 %   \[\nabla_{e_1}(e_1) = e_1,\quad  \nabla_{e_1}(e_2) = e_2, \quad  \nabla_{e_2}(e_1) = 0, \quad \nabla_{e_2}(e_2) = 0.\]
    
    \item \underline{The Lie superalgebra $D^{7}_{pq}$, with $p=-q, q\neq 0, -1 $:} This Lie superalgebra is a $T^*$-extension of the Lie superalgebra $\mathfrak{h}=\mathrm{Span}\{e_1 | e_4\}$ with the bracket $[e_1,e_4]_\fh = -pe_4$, and where $\mathfrak{h}^*$ is spanned by $e_2 = e_1^{*}|e_3 = e_4^{*}$. The 2-form $\alpha \equiv 0$ and the connection $\nabla$ is given by
    
    \[\nabla_{e_1}(e_1) = -e_1,\quad  \nabla_{e_1}(e_4) = -pe_4, \quad  \nabla_{e_4}(e_1) = 0, \quad \nabla_{e_4}(e_4) = 0.\]

    \item \underline{The Lie superalgebra $D^{7}_{pq}$, with $p=-q, q= -1 $:} In the case where the form is odd, then this Lie superalgebra is a $\Pi T^*$-extension of the Lie superalgebra $\mathfrak{h}=\mathrm{Span}\{e_1 ,  e_2\}$ with the bracket $[e_1,e_2]_\fh = e_2$, and where $\Pi \mathfrak{h}^*$ is spanned by $e_3 = \Pi(e_1^{*}) ,  e_4 = \Pi(e_2^{*})$. The 2-form $\beta \equiv 0$ and the connection $\nabla$ is given by
    
    \[\nabla_{e_1}(e_1) = -e_1,\quad  \nabla_{e_1}(e_2) = e_2, \quad  \nabla_{e_2}(e_1) = 0, \quad \nabla_{e_2}(e_2) = 0.\]

    In the case where the form is even, then this Lie superalgebra is a $ T^*$-extension of the Lie superalgebra $\mathfrak{h}=\mathrm{Span}\{e_1 | e_4\}$ with the bracket $[e_1,e_4]_\fh = -e_4$, and where $\mathfrak{h}^*$ is spanned by $e_2 = e_1^{*} | e_3 = e_4^{*}$. The 2-form $\alpha \equiv 0$ and the connection $\nabla$ is given by
    
    \[\nabla_{e_1}(e_1) = -e_1,\quad  \nabla_{e_1}(e_4) = -e_4, \quad  \nabla_{e_4}(e_1) = 0, \quad \nabla_{e_4}(e_4) = 0.\]

    \item \underline{The Lie superalgebra $D^{8}_{p}$, with $p\neq 0$:} This Lie superalgebra is not quasi-Frobenius. 

    \item \underline{The Lie superalgebra $D^{8}_{-1}$:} This Lie superalgebra is a $\Pi T^*$-extension of the Lie superalgebra $\mathfrak{h}=\mathrm{Span}\{e_1 ,  e_2\}$ with the bracket $[e_1,e_2]_\fh = e_2$, and where $\Pi \mathfrak{h}^*$ is spanned by $e_3 = \Pi(e_1^{*}) ,  e_4 = \Pi(e_2^{*})$. The 2-form $\beta \equiv 0$ and the connection $\nabla$ is given by
    
    \[\nabla_{e_1}(e_1) = e_1-e_2,\quad  \nabla_{e_1}(e_2) = e_2, \quad  \nabla_{e_2}(e_1) = 0, \quad \nabla_{e_2}(e_2) = 0.\]

    \item \underline{The Lie superalgebra $D^{9}_{pq}$, with $q> 0$:} This Lie superalgebra is not quasi-Frobenius. 

    \item \underline{The Lie superalgebra $D^{10}_{q}$ with $q\not =-1$:} This Lie superalgebra is a $\Pi T^*$-extension of the Lie superalgebra $\mathfrak{h}=\mathrm{Span}\{e_1 | e_4\}$ with the bracket $[e_1,e_4]_\fh = qe_4$, and where $\Pi \mathfrak{h}^*$ is spanned by $e_2 = \Pi(e_4^{*}) | e_3 =-(q+1) \Pi(e_1^{*})$. The 2-form $\beta \equiv 0$ and the connection $\nabla$ is given by
    
    \[\nabla_{e_1}(e_1) = -(q+1)e_1,\quad  \nabla_{e_1}(e_4) = -e_4, \quad  \nabla_{e_4}(e_1) =-(q+1) e_4, \quad \nabla_{e_4}(e_4) =0.\]

    \item \underline{The Lie superalgebra $(D^{7}_{1/2\; 1/2})^1$:} This Lie superalgebra does not admit a 2-dimensional Lagrangian ideal. Indeed, let us suppose the contrary, and let us  denote by $\mathfrak{a}=\mathrm{Span}\{X,Y\}$ such an ideal. Let us write $X=x_1e_1+x_2e_2+x_3e_3+x_4e_4$ and $Y=y_1e_1+y_2e_2+y_3e_3+y_4e_4$. 
    We compute the brackets  $[e_i,X]$,  for all $i\in \{1,2,3,4\}$. we obtain 
    \[ 
    \begin{array}{ll}
    [e_1, X]=x_2e_2 + \dfrac{1}{2}x_3 e_3+ \dfrac{1}{2}x_4 e_4, &  [e_2, X]=-x_1e_2,\\[2mm]
    [e_3,X]=-\dfrac{1}{2}x_1 e_3+x_3 e_2, & [e_4, X]=-\dfrac{1}{2}x_1 e_4+x_4 e_2.
    \end{array}
    \]
    If $x_1 \neq 0 $, from the second bracket we get that $e_2 \in \mathfrak{a}$, followed by $e_3 \in \mathfrak{a}$ and $e_4 \in \mathfrak{a}$, which implies that the ideal is not of dimension 2. We are in the case where $x_1=0$ (and similarly $y_1=0$).
    \[
    \begin{array}{ll}
    [e_1, X]=x_2e_2 + \dfrac{1}{2}x_3 e_3+ \dfrac{1}{2}x_4 e_4, &  [e_2, X]=0,\\[2mm]
    [e_3,X]=x_3 e_2, & [e_4, X]=x_4 e_2.
    \end{array}
    \]
    
    We consider the following cases:
    \begin{enumerate}
        \item If $x_3=x_4=0$. In this case,  $e_2 \in \mathfrak{a}$ and we can write $X=e_2$ and $Y=y_2e_2+y_3e_3+y_4e_4$. We  can simplify and write $Y=z_3e_3+z_4e_4$. Since\footnote{As in \cite{BM, BEM, BE}, we adopt the following convention: $\langle e_i^*, e_j\rangle=\delta_{ij}$ and  $\langle e_i^*\otimes e_j^*, e_k\otimes e_l\rangle= (-1)^{|e_k||e_j|} \langle e_i^*, e_k \rangle \langle e_j^*, e_l \rangle$.} $\omega(Y,Y)=z_3^2+z_4^2\not =0,$ it follows that $\mathfrak{a}$ is not Lagrangian. 
        
        \item If $x_3=0,$ and $ x_4\neq 0$.  Looking at brackets again we get $[e_1, X]=x_2e_2 + \dfrac{1}{2}x_4 e_4,$ and $ [e_4, X]=x_4 e_2$, which implies $e_2, e_4 \in \mathfrak{a}$, so $\mathfrak{a}=\mathrm{Span}\{ e_2, e_4\} $. We have to check that $\mathfrak{a}$ is not Lagrangian. Indeed, $\omega(e_4, e_4)=1$.
        
        \item If $x_3 \neq 0,$ and $ x_4= 0$. This case will be  similar to the previous one since the equations are symmetrical in $x_3$ and $x_4$, which implies that no Lagrangian ideal exists.  
        
        \item If $x_3 \neq 0,$ and $ x_4\neq 0$. Since all the previous cases showed that we cannot find a Lagrangian ideal, and since the proof is symmetrical in X and Y without loss of generality we can state that $y_3 \neq 0, y_4\neq 0$. This implies that $e_2 \in \mathfrak{a}$ and that $x_2 e_2+\dfrac{1}{2}x_3 e_3 +\dfrac{1}{2}x_4 e_4 \in \mathfrak{a}$. We can write $X=e_2$ and $Y=z_3e_3+z_4e_4$. Once again, $\omega(Y,Y)=z_3^2+z_4^2\not =0$ which implies that $\mathfrak{a}$ is not Lagrangian.
    \end{enumerate}

    \item \underline{The Lie superalgebra $(D^{7}_{1/2\; 1/2})^2$:} This Lie superalgebra is a $T^*$-extension of the Lie superalgebra $\mathfrak{h}=\mathrm{Span}\{e_1 | x:=e_3-e_4\}$ with the bracket $[e_1,x]_\fh = \dfrac{1}{2}x$, and where $\mathfrak{h}^*$ is spanned by $e_2 = e_1^{*} | e_3+e_4 = -2x^{*}$. The 2-form $\alpha \equiv 0 $ and the connection $\nabla$ is given by
    
    \[\nabla_{e_1}(e_1) = -e_1,\quad  \nabla_{e_1}(x) = - \frac{1}{2}x, \quad  \nabla_{x}(e_1) = -x, \quad \nabla_{x}(x) = 0.\]

    \item \underline{The Lie superalgebra $(D^{7}_{1/2\; 1/2})^3$:} This Lie superalgebra is not quasi-Frobenius. 
    
    \item \underline{The Lie superalgebra $(D^{7}_{1-p,p})$, with $p\le \dfrac{1}{2}$:} This Lie superalgebra is a $T^*$-extension of the Lie superalgebra $\mathfrak{h}=\mathrm{Span}\{e_1 | e_4\}$ with the bracket $[e_1,e_4]_\fh = (1-p)e_4$, and where $\mathfrak{h}^*$ is spanned by $e_2 = - e_1^{*} | e_3 = e_4^{*}$. The 2-form $\alpha \equiv 0$ and the connection $\nabla$ is given by
    
    \[\nabla_{e_1}(e_1) =- e_1,\quad  \nabla_{e_1}(e_4) = - p e_4, \quad  \nabla_{e_4}(e_1) = - e_4, \quad \nabla_{e_4}(e_4) = 0.\]

    \item \underline{The Lie superalgebra $(D^{8}_{1/2})$:} This Lie superalgebra is not quasi-Frobenius. 

    \item \underline{The Lie superalgebra $(D^{9}_{1/2, p})$, with $p>0$:} This Lie superalgebra does not admit a 2-dimensional Lagrangian ideal. Indeed, let us suppose there exists one, denoted by $\mathfrak{a}=\mathrm{Span}\{X,Y\}$. Let us write $X=x_1e_1+x_2e_2+x_3e_3+x_4e_4$ and $Y=y_1e_1+y_2e_2+y_3e_3+y_4e_4$. 
    We compute the brackets  $[e_i,X]$,  for all $i\in \{1,2,3,4\}$. We obtain 
    \[ 
    \begin{array}{l}
    [e_1, X]=x_2e_2 + \dfrac{1}{2}x_3 e_3 - x_3pe_4+x_4pe_3+\dfrac{1}{2}x_4e_4, \\[2mm]
    [e_2, X]=-x_1e_2,\\[2mm]
    [e_3,X]=-\dfrac{1}{2}x_1 e_3+px_1e_4+x_3 e_2,\\[3mm]
    [e_4, X]=-\dfrac{1}{2}x_1 e_4-px_1e_3+x_4 e_2.
    \end{array}
    \]
    We will analyze a few cases. If $x_1 \neq 0$, given the second bracket we get that $e_2 \in \mathfrak{a}$ and according to the last brackets $e_3 \in \mathfrak{a}$ and $e_4 \in \mathfrak{a}$. This case is not possible since $\mathfrak{a}$ has dimension 2. 
    This implies that $x_1=0$. We get then
    \[
    \begin{array}{ll}
    [e_1, X]=x_2e_2 + \dfrac{1}{2}x_3 e_3 - x_3pe_4+x_4pe_3+\dfrac{1}{2}x_4e_4, & [e_2, X]=0,\\[2mm]
    [e_3,X]=x_3 e_2, & [e_4, X]=x_4 e_2.
    \end{array}
    \]
    For all the cases $x_3, x_4 =0$ or $\neq 0$ we get that $e_2 \in \mathfrak{a}$. We can write $X=e_2$ and $Y=z_3e_3+z_4e_4$. As $\omega(Y,Y)\not = 0$, it follows that  $\mathfrak{a}$ is not Lagrangian. 
    
    Thus, this Lie superalgebra does not admit a 2-dimensional Lagrangian ideal.
     
    \item \underline{The Lie superalgebra $(D^{10}_{0})^1$:} In the case where the form is odd, this Lie superalgebra is a $\Pi T^*$-extension of the Lie superalgebra $\mathfrak{h}=\mathrm{Span}\{e_1 | e_4\}$ with the bracket $[e_4,e_4]_\fh = e_1$, and where $\mathfrak{h}^*$ is spanned by $e_2 = - \Pi(e_4^{*}) | e_3 = \Pi(e_1^{*})$. The 2-form $\beta \equiv 0$ and the connection $\nabla$ is given by
    
    \[\nabla_{e_1}(e_1) =- e_1,\quad  \nabla_{e_1}(e_4) = -  e_4, \quad  \nabla_{e_4}(e_1) = - e_4, \quad \nabla_{e_4}(e_4) = \frac{1}{2}e_1.\]
    In the case where the form is even, this Lie superalgebra is a $T^*$-extension of the Lie superalgebra $\mathfrak{h}=\mathrm{Span}\{e_1 | e_4\}$ with the bracket $[e_4,e_4]_\fh = e_1$, and where $\mathfrak{h}^*$ is spanned by $e_2 = 2 e_1^{*} | e_3 = e_4^{*}$. The 2-form $\alpha \equiv 0$ and the connection $\nabla$ is given by
    
    \[\nabla_{e_1}(e_1) =- e_1,\quad  \nabla_{e_1}(e_4) = -  e_4, \quad  \nabla_{e_4}(e_1) = - e_4, \quad \nabla_{e_4}(e_4) = \frac{1}{2}e_1.\]

    \item \underline{The Lie superalgebra $(D^{10}_{0})^2$:} In the case where the form is odd, this Lie superalgebra is a $\Pi T^*$-extension of the Lie superalgebra $\mathfrak{h}=\mathrm{Span}\{e_1 | e_4\}$ with the bracket $[e_4,e_4]_\fh = -e_1$, and where $\mathfrak{h}^*$ is spanned by $e_2 = - \Pi(e_4^{*}) | e_3 = \Pi(e_1^{*})$. The 2-form $\beta \equiv 0$ and the connection $\nabla$ is given by
    
    \[\nabla_{e_1}(e_1) =- e_1,\quad  \nabla_{e_1}(e_4) = -  e_4, \quad  \nabla_{e_4}(e_1) = - e_4, \quad \nabla_{e_4}(e_4) = -\frac{1}{2}e_1.\]
    In the case where the form is even, this Lie superalgebra is a $T^*$-extension of the Lie superalgebra $\mathfrak{h}=\mathrm{Span}\{e_1 | e_4\}$ with the bracket $[e_4,e_4]_\fh = -e_1$, and where $\mathfrak{h}^*$ is spanned by $e_2 = 2 e_1^{*} | e_3 = e_4^{*}$. The 2-form $\alpha   \equiv 0$ and the connection $\nabla$ is given by
    
    \[\nabla_{e_1}(e_1) =- e_1,\quad  \nabla_{e_1}(e_4) = -  e_4, \quad  \nabla_{e_4}(e_1) = - e_4, \quad \nabla_{e_4}(e_4) = -\frac{1}{2}e_1.\]

    \item \underline{The Lie superalgebra $(2A_{1,1}+2A)^1$:} This Lie superalgebra is not quasi-Frobenius. 
    \item \underline{The Lie superalgebra $(2A_{1,1}+2A)^2$:} This Lie superalgebra is not quasi-Frobenius. 
    \item \underline{The Lie superalgebra $(2A_{1,1}+2A)^3_p$, for $p=\frac{1}{2}$:} This Lie superalgebra is a $\Pi T^*$-extension of the abelian Lie superalgebra $\mathfrak{h}=\mathrm{Span}\{e_3 , e_4\}$, and where $\Pi \mathfrak{h}^*$ is spanned by $e_1 = -\Pi(e_4^*) , \;  e_2 = \Pi(e_3^{*})$. The connection $\nabla$ is trivial, and the 2-form $\beta$ is given by  
    \[
\beta=\frac{1}{2} \Pi (e_4^*)\otimes  e_3^* \wedge e_3^* -\frac{1}{2} \Pi (e_3^*)\otimes  e_4^* \wedge e_4^*- \frac{1}{2} (\Pi(e_3^*)- \Pi(e_4^*) )\otimes e_3^*\wedge e_4^*.
    \]
    \item \underline{The Lie superalgebra $(2A_{1,1}+2A)^4_p$:} This Lie superalgebra is not quasi-Frobenius. 

    \item \underline{The Lie superalgebra $(C^1_1+A)$:} This Lie superalgebra is a $ T^*$-extension of the abelian Lie superalgebra $\mathfrak{h}=\mathrm{Span}\{e_1 | x:=e_3-2e_4\}$, and where $\mathfrak{h}^*$ is spanned by $e_2 = \frac{1}{2}e_1^{*} | e_3 = x^{*}$. The 2-form $\alpha $ is given by:
    \[ \alpha = e^*_1 \otimes (x^* \wedge x^*) +x^* \otimes (e^*_1 \wedge x^*)\]
    The connection $\nabla$ is given by
    
    \[\nabla_{e_1}(e_1) = -e_1,\quad  \nabla_{e_1}(x) = -x, \quad  \nabla_{x}(e_1) = -x, \quad \nabla_{x}(x) = 0.\]
    
    \item \underline{The Lie superalgebra $(C^1_{1/2}+A)$:} This Lie superalgebra does not admit a Lagrangian ideal of dimension 2. Indeed, let us suppose there is one, denoted by  $\mathfrak{a}=\mathrm{Span}\{X,Y\}$. Let us write $X=x_1e_1+x_2e_2+x_3e_3+x_4e_4$ and $Y=y_1e_1+y_2e_2+y_3e_3+y_4e_4$. 
    We compute the brackets  $[e_i,X]$,  for all $i\in \{1,2,3,4\}$. We obtain
    \[ 
    \begin{array}{l}
    [e_1, X]=x_2e_2 + \dfrac{1}{2}x_3 e_3, \quad  [e_2, X]=-x_1e_2, \quad 
    [e_3,X]=-\dfrac{1}{2}x_1 e_3+x_3 e_2, \quad  [e_4, X]=0.
    \end{array}
    \]
    We will analyze a few cases. 
    
    If $x_1 \neq 0$, given the second bracket implies that $e_2 \in \mathfrak{a}$ and according to the third bracket $e_3 \in \mathfrak{a}$, which implies that $\mathfrak{a}=\mathrm{Span}\{e_2, e_3\}$. Since $\omega(e_3,e_3)\not =0$, it follows that $\mathfrak a$ is not Lagrangian. 
    
    If $x_1=0$, and $x_3 \neq 0$, we get that $e_2 \in \mathfrak{a}$, followed by $e_3 \in \mathfrak{a}$, so $\mathfrak{a}=\mathrm{Span}\{e_2, e_3\}$,  which we showed is not possible.
    It follows that $x_3 =0$, and using the same reasoning as above, we also conclude  that for $Y$, the scalars $y_1=y_3=0$. Looking at the brackets, we either have $x_2=0$ or $e_2 \in \mathfrak{a}$. 
    If $x_2=0$, then $X=x_4 e_4$, and $Y=y_2 e_2 +y_4 e_4$. We can simplify and write $X=e_4$, and $Y=e_2$. Since $\omega(e_4,e_4)\not =0$, it follows that  $\mathfrak{a}$ is not Lagrangian. On the other hand,  if $x_2\neq 0$, then $X=x_2e_2+ x_4 e_4$, and $Y=y_2 e_2 +y_4 e_4$. We can simplify and write $X=e_2+z_4 e_4$, and $Y= e_4$. Again, $\mathfrak{a}$  is not Lagrangian. 
    
    Thus, this Lie superalgebra does not admit a 2-dimensional Lagrangian ideal.
    
    \item \underline{The Lie superalgebra $(C^2_{-1}+A)$:} This Lie superalgebra is not quasi-Frobenuis. 
    \item \underline{The Lie superalgebra $(C^3+A)$:} This Lie superalgebra is a $ T^*$-extension of the Lie superalgebra $\mathfrak{h}=\mathrm{Span}\{e_2 | e_4\}$ with the bracket $[e_4,e_4]_{\mathfrak h} = e_2$, and where $\mathfrak{h}^*$ is spanned by $e_1 = 2 e_2^{*}\, | \, e_3 = e_4^{*}$. The 2-form $\alpha \equiv 0$ and the connection $\nabla$ is given by
    \[\nabla_{e_2}(e_2) = 0,\quad  \nabla_{e_2}(e_4) =0 , \quad  \nabla_{e_4}(e_2) = 0, \quad \nabla_{e_4}(e_4) = \frac{1}{2}e_2.\]
    
    \item \underline{The Lie superalgebra $(C^5_{0}+A)$:} This Lie superalgebra is not quasi-Frobenuis. 

    \item \underline{The Lie superalgebra  $D^1$:} This Lie superalgebra does not admit a non-degenerate homogeneous closed form.

    \item \underline{The Lie superalgebra  $D^2_q, \text{ with } q \not =-1,0 $:} This Lie superalgebra is indeed not quasi-Frobenuis. 

    \item \underline{The Lie superalgebra  $D^2_{-1}$:} This Lie superalgebra does not admit a non-degenerate homogeneous closed form.

    \item \underline{The Lie superalgebra  $D^3_{pq}$:} This Lie superalgebra is not quasi-Frobenuis. 

    \item \underline{The Lie superalgebra  $D^{11}_{pq} \text{ with } 0< |p|\leq |q|\leq 1$:} This Lie superalgebra is not quasi-Frobenuis. 

    \item \underline{The Lie superalgebra  $D^{11}_{pq} \text{ with } (p,q)=(-1,-1)$:} This Lie superalgebra does not admit a non-degenerate homogeneous closed form.

    \item \underline{The Lie superalgebra  $D^{11}_{pq} \text{ with } 0< |p|\leq 1 , p=-q $:} This Lie superalgebra does not admit a non-degenerate homogeneous closed form.

    \item \underline{The Lie superalgebra  $D^{11}_{pq} \text{ with } 0< |p|< 1, q=-1$:} This Lie superalgebra does not admit a non-degenerate homogeneous closed form.

    \item \underline{The Lie superalgebra  $D^{12}$:} This Lie superalgebra does not admit a non-degenerate homogeneous closed form.

    \item \underline{The Lie superalgebra  $D^{13}_p,\text{ with } p\not =0$:} This Lie superalgebra is not quasi-Frobenuis. 

    \item \underline{The Lie superalgebra  $D^{13}_{-1}$:} This Lie superalgebra does not admit a non-degenerate homogeneous closed form.

    \item \underline{The Lie superalgebra  $D^{14}_{pq}, \text{ with } p\not =0, q\geq 0$:} This Lie superalgebra is not quasi-Frobenuis. 

    \item \underline{The Lie superalgebra  $D^{15}$:} This Lie superalgebra does not admit a non-degenerate homogeneous closed form.

    \item \underline{The Lie superalgebra  $D^{16}$:} This Lie superalgebra is not quasi-Frobenuis. 

    \item \underline{The Lie superalgebra  $(A_{3,1}+A)$:} This Lie superalgebra is not quasi-Frobenuis. 

    \item \underline{The Lie superalgebra  $(D^{3}_{p,-1/2}), \text{ with } p\not =0$:} This Lie superalgebra is not quasi-Frobenuis. 

    \item \underline{The Lie superalgebra  $(D^3_{-1/2, -1/2})$:} This Lie superalgebra does not admit a non-degenerate homogeneous closed form.
    
    \item \underline{The Lie superalgebra  $(D^2_{-1/2})^1$:} This Lie superalgebra is  not quasi-Frobenuis. 
    \item \underline{The Lie superalgebra   $(D^2_{-1/2})^2$:} This Lie superalgebra is not quasi-Frobenuis. 
    \item \underline{The Lie superalgebra  $(A_{1,1}+3 A)^1$:} This Lie superalgebra is not quasi-Frobenuis. 
    \item \underline{The Lie superalgebra   $(A_{1,1}+3 A)^2$:} This Lie superalgebra is not quasi-Frobenuis. 
\end{enumerate}
\vspace{3mm}
Tables~\ref{tab:NoQF}--\ref{tab:PTQF} below summarize the outcome of the classification. 
Among the Lie superalgebras considered, 22 do not admit a quasi-Frobenius structure, 
whereas 9 admit a nonhomogeneous structure. 
Of the quasi-Frobenius cases, 4 do not arise as Lagrangian extensions. 
Moreover, 8 are $T^*$-extensions and 10 are $\Pi T^*$-extensions. It is worth mentioning that three Lie superalgebras admit both a closed orthosymplectic and periplectic form, namely 
$(D_0^{10})^1$, $(D_0^{10})^2$, and $D_{pq}^{7}$ with $(p,q)=(1,-1)$.
%%%%%%%%%%%%%%%%%%%%%%%%%%%%%
\begin{table}[ht]
\centering
\renewcommand{\arraystretch}{1.2}
\begin{tabular}{lll}
\hline
\multicolumn{3}{c}{Lie superalgebras without quasi-Frobenius structures} \\
\hline
$(D_{1/2\; 1/2}^7)^3$ & $(D^8_{1/2})$  &  
$(2A_{1,1}+2A)^1$ \\[2mm] 
$(2A_{1,1}+2A)^2$ & 
$D^7_{pq}$ & $(A_{1,1}+3A)^1$\\[2mm]
 $(A_{1,1}+3A)^2$&
$(D^2_{-1/2})^1$ & 
$(D^2_{-1/2})^2$ \\[2mm] $(D^3_{p,-1/2}), \, p\not =0$& 
$(A_{3,1}+A)$&  
$D^{16}$\\[2mm]  
$D^{14}_{pq}, \, p\not =0, q\geq 0$& 
$D^{13}_p, \, p\not = 0$
& $D^3_{pq}$\\[2mm]
$D^{11}_{pq}, \, 0<|p|\leq |q|\leq 1$& 
$D^2_q, \, q\not = -1,0$
& 
$D^9_{pq},\, q>0$\\[2mm] 
$D^8_p, \, p\not =0$
&
$(C_0^5+A)$ &
 $(C^2_{-1}+A)$\\[2mm]  $(2A_{1,1}+2A)^4_p$  & &\\
\hline
\end{tabular}
\caption{No quasi-Frobenius structures}
\label{tab:NoQF}
\end{table}

% \begin{table}[h!]
% \centering
% \renewcommand{\arraystretch}{1.2}
% \caption{Lie superalgebras that do not admit a quasi-Frobenius structure.}
% \label{tab:nonQF}
% \begin{tabular}{l}
% \hline
% \begin{minipage}{0.9\linewidth}
% \begin{multicols}{3}
% \begin{itemize}
% \setlength{\itemsep}{0pt}
% \item $(D_{1/2\; 1/2}^7)^3$
% \item $(D^8_{1/2})$  
% \item $(2A_{1,1}+2A)^1$ \item $(2A_{1,1}+2A)^2$
% \item $D^7_{pq}$ \item $(A_{1,1}+3A)^1$
% \item $(A_{1,1}+3A)^2$
% \item $(D^2_{-1/2})^1$
% \item $(D^2_{-1/2})^2$
% \item $(D^3_{p,-1/2}), \, p\not =0$\item $(A_{3,1}+A)$\item 
% $D^{16}$\item  $D^{14}_{pq}, \, p\not =0, q\geq 0$
% \item $D^{13}_p, \, p\not = 0$
% \item $D^3_{pq}$
% \item $D^{11}_{pq}, \, 0<|p|\leq |q|\leq 1$
% \item 
% $D^2_q, \, q\not = -1,0$
% \item 
% $D^9_{pq},\, q>0$
% \item 
% $D^8_p, \, p\not =0$
% \item $(C_0^5+A)$
% \item $(C^2_{-1}+A)$
% \item $(2A_{1,1}+2A)^4_p$. 
% % continue the list
% \end{itemize}
% \end{multicols}
% \end{minipage}
% \\
% \hline
% \end{tabular}
% \end{table}
%%%%%%%%%%%%%%%%%%%%%%%%%%%%%%%
\begin{table}[ht]
\centering
\renewcommand{\arraystretch}{1.2}
\begin{tabular}{lll}
\hline
\multicolumn{3}{c}{Lie superalgebras with a {\bf nonhomogeneous} quasi-Frobenius structure} \\
\hline
$(D^3_{-1/2\; 1/2})$ &  $D^{15}$&  $D^{13}_{-1}$\\[2mm] $D^{12}$ &  $D^{11}_{pq}, \, 0<|p|<1,\, q=-1$&  $D^{11}_{pq}, \, 0<|p|<1,\, p=-q$\\[2mm] 
$D^{11}_{pq}, \, p=q=-1$ & $D^{2}_{-1}$& $D^{1}$. \\
% continue the list
\hline
\end{tabular}
\caption{Nonhomogeneous}
\label{tab:NHQF}
\end{table}
%%%%%%%%%%%%%%%%%%%%%%%%%%%%%%%
\begin{table}[ht]
\centering
\renewcommand{\arraystretch}{1.2}
\begin{tabular}{lll}
\hline
\multicolumn{3}{c}{Quasi-Frobenuis Lie superalgebras that are {\bf not} Lagrangian extensions} \\
\hline
$(C_{1/2}^1+A)$ & $(D^9_{1/2\; p}), \, p>0$ & $(D^7_{1/2\; 1/2})^1$\\[2mm] $D^6$ & & \\
% continue the list
\hline
\end{tabular}
\caption{Not Lagrangian extensions}
\label{tab:NLQF}
\end{table}
%%%%%%%%%%%%%%%%%%%%%%%%%%%%%%%
%%%%%%%%%%%%%%%%%%%%%%%%%%%%%%%
\begin{table}[ht]
\centering
\renewcommand{\arraystretch}{1.2}
\begin{tabular}{lll}
\hline
\multicolumn{3}{c}{Quasi-Frobenuis Lie superalgebras that are  $T^*$-extensions} \\
\hline
$(C^3+A)$ &   $(C_1^1+A)$ &  $(D_0^{10})^1, \, (|\omega|=\bar 0)$\\[2mm]  
$(D_0^{10})^2, \, (|\omega|=\bar 0)$&  $D^7_{1-p,p}, \, p\leq \frac{1}{2}$ & $(D^7_{1/2 \; 1/2})^2$\\[2mm] 
$D^7_{pq},\, p=-q, q\not =0,-1$ & $D^7_{pq},\, p=-q=1$ $(|\omega|=\bar 0)$ &   \\
% continue the list
\hline
\end{tabular}
\caption{$T^*$-extensions}
\label{tab:TQF}
\end{table}
%%%%%%%%%%%%%%%%%%%%%%%%%%%%%
\begin{table}[ht]
\centering
\renewcommand{\arraystretch}{1.2}
\begin{tabular}{lll}
\hline
\multicolumn{3}{c}{Quasi-Frobenuis Lie superalgebras that are  $\Pi T^*$-extensions} \\
\hline 
$(2A_{1,1}+2A)^3_p, \, p=\frac{1}{2}$&  $(D_0^{10})^2, \, (|\omega|=\bar 1)$ &  $(D_0^{10})^1,\, (|\omega|=\bar 1)$\\ 
$D^{10}_q, \, q \not =-1$ &  $D^6$ &  
$D^7_{-1,q},\, q\leq 1$\\[2mm] $D^7_{p,q}, \, p=-q=1,\, (|\omega|=\bar 1)$ &  $D^5$ &  $D^8_{-1}$\\[2mm] 
$D^7_{p,q}, \, p=-q=1$  & &\\
% continue the list
\hline
\end{tabular}
\caption{$\Pi T^*$-extensions}
\label{tab:PTQF}
\end{table}

\newpage 

%%%%%%%%%%%%%%%%%%%%%%%%%%%%%%%%%%%%%%%%%%%%
\section{Left-symmetric structures on the 4-dimensional real Lie superalgebras}\label{leftsym}
%%%%%%%%%%%%%%%%%%%%%%%%%%%%%%%%%%%%%%%%%%%%%
In the list of {\it indecomposable} 4-dimensional real Lie superalgebras, we will list left-symmetric structures that are compatible with the Lie structure. With the exception of $(D_0^{10})^1$ and $(D_0^{10})^2$, all of these structures are Novikov. Moreover, we show that each of these Lie superalgebras admit a Balinsky-Novikov structure. In line with standard practice, only non-zero products are displayed. 
\begin{enumerate}
    \item \underline{The Lie superalgebra $D^5$}: A Novikov structure is given by 
    \[
    e_1\cdot e_3=  e_3,   \quad e_1\cdot e_4=e_4, \quad e_2\cdot e_4=e_3. 
    \]
    Additionally, this structure is also Balinsky-Novikov. As a matter of fact, the non-zero associators are given by
    \[
    (e_1, e_1, e_4)=-e_4, \quad  (e_2, e_1, e_4)=- e_3, \quad (e_1,e_1,e_3)=-e_3, \quad (e_1, e_2, e_4)=-e_3.
    \]
    \item \underline{The Lie superalgebra $D^6$}: A Novikov structure is given by
    \[
e_1\cdot e_3=e_3, \quad e_1\cdot e_4=e_4, \quad e_2\cdot e_3=-e_4, \quad e_2\cdot e_4=e_3.
    \]
     Additionally, this structure is also Balinsky-Novikov.
    \item \underline{The Lie superalgebra $D^7_{pq}$ with ($p\geq q, \; pq\not=0$):} A Novikov structure, depending on a parameter $\gamma \in \mathbb R$, is given by
     \[
e_1\cdot e_1= \gamma e_2, \quad e_1\cdot e_2=e_2, \quad e_1\cdot e_3=p e_3, \quad e_1\cdot e_4=q e_4.
\]
 Additionally, this structure is also Balinsky-Novikov.
 \item \underline{The Lie superalgebra $D^8_{p}$ with $ p\not=0$:} A Novikov structure, depending on a parameter $\gamma \in \mathbb R$, is given by
    \[
    e_1\cdot e_1= \gamma e_2, \quad e_1\cdot e_2=e_2, \quad e_1\cdot e_3=p  e_3, \quad e_1\cdot e_4=e_3+ p e_4.
    \]
     Additionally, this structure is also Balinsky-Novikov.
     \item \underline{The Lie superalgebra $D^9_{pq}$ with $ q>0$:} A Novikov structure, depending on a parameter $\gamma \in \mathbb R$, is given by
      \[
      \begin{array}{llll}
    e_1\cdot e_1= \gamma e_1, &  e_1\cdot e_2=(1+\gamma)e_2, &  e_1\cdot e_3=(p+\gamma)  e_3-q e_4, & e_1\cdot e_4=q e_3+ (p+\gamma) e_4, \\[2mm]
    e_2\cdot e_1=\gamma e_2, &  e_3\cdot e_1=\gamma e_3, &  e_4\cdot e_1=\gamma e_4.
    \end{array}
    \]
    A Balinsky-Novikov structure, depending on a parameter $\gamma \in \mathbb R$, is given by
\[
\begin{array}{lllll}
e_1 \cdot e_1= \gamma e_2, & e_1 \cdot e_2=e_2, & e_1 \cdot e_3=p e_3 -q e_4,&  
e_1 \cdot e_4= q e_3 + p e_4.
\end{array}
\]
\item \underline{The Lie superalgebra $D^{10}_{q}:$} A Novikov structure, depending on a parameter $\gamma \in \mathbb R$, is given by
\[
\begin{array}{llll}
e_1\cdot e_1=-q e_1 +\gamma(q-1)e_2, &  e_1\cdot e_2=(1-q)e_2, & e_1\cdot e_3=e_3, & e_1\cdot e_4=\gamma e_3, \\[2mm] e_2\cdot e_1=-q e_2, & e_3\cdot e_1=-q e_3, &  e_4 \cdot e_1=\gamma e_3 -q e_4, & e_4 \cdot e_2=-e_3.
\end{array}
\]
A Balinsky-Novikov structure is given by
\[
\begin{array}{llll}
e_1\cdot e_1= -e_1, & e_1 \cdot e_3= \frac{1}{2}(1+2q)e_3, & e_1 \cdot e_4=\frac{1}{2}(2q-1)e_4, & 
e_2 \cdot e_1= -e_2, \\[2mm] 
e_2 \cdot e_4=e_3, & e_3 \cdot e_1=-e_3, & e_4\cdot e_1=-e_4.
\end{array}
\]
 \item \underline{The Lie superalgebra $(D^{7}_{1/2\; 1/2})^1:$} A Novikov structure, depending on parameters $\lambda, \gamma \in \mathbb R$, is given by
 \[
 \begin{array}{lllll}
e_1\cdot e_1=-\frac{1}{2} e_1 +\lambda e_2, &  e_1\cdot e_2=\frac{1}{2} e_2, &  e_2\cdot e_1= - \frac{1}{2} e_2, & e_3\cdot e_1=- \frac{1}{2} e_3, & e_3\cdot e_3=\frac{1}{2} e_2\\[2mm]
e_3\cdot e_4= \gamma e_2, &  e_4 \cdot e_1=-\frac{1}{2} e_4, &  e_4 \cdot e_3=- \gamma e_2, & e_4\cdot e_4=\frac{1}{2}e_2.
\end{array}
 \]
 A Balinsky-Novikov structure, depending on parameters $\lambda, \gamma \in \mathbb R$, is given by
 \[
 \begin{array}{llll}
e_1 \cdot e_1= (\gamma -1)e_1+\lambda e_2, & e_1\cdot e_2= \gamma e_2, & e_1\cdot e_3=\frac{\gamma}{2}e_3, & e_1\cdot e_4= \frac{\gamma}{2}e_4,\\[2mm]
e_2\cdot e_1=(\gamma-1)e_2, & e_3\cdot e_1=(\gamma-1)e_3, & e_3\cdot e_3= e_2, & e_4\cdot e_1=(\gamma -1)e_4,\\[2mm]
e_4\cdot e_4=e_2.
 \end{array}
 \]
 \item \underline{The Lie superalgebra $(D^{7}_{1/2\; 1/2})^2:$} A Novikov structure, depending on a parameter $\lambda \in \mathbb R$, is given by
 \[
 \begin{array}{llll}
e_1\cdot e_1=-\frac{1}{2} e_1 +\lambda e_2, &  e_1\cdot e_2=\frac{1}{2} e_2, &  e_2\cdot e_1= - \frac{1}{2} e_2,& e_3\cdot e_1=- \frac{1}{2} e_3, \\[2mm] e_3\cdot e_3=\frac{1}{2} e_2,& 
e_4 \cdot e_1=-\frac{1}{2} e_4, &  e_4\cdot e_4=-\frac{1}{2}e_2.
\end{array}
 \]
A Balinsky-Novikov structure, depending on parameters $\lambda, \gamma \in \mathbb R$,  is given by 
 \[
 \begin{array}{llll}
e_1 \cdot e_1= \gamma e_1+\lambda e_2, & e_1\cdot e_2= (1+\gamma) e_2, & e_1\cdot e_3=\frac{1}{2}(1+\gamma )e_3, & e_1\cdot e_4= \frac{1}{2}(1+\gamma)e_4,\\[2mm]
e_2\cdot e_1=\gamma e_2, & e_3\cdot e_1=\gamma e_3, & e_3\cdot e_3= e_2, & e_4\cdot e_1=\gamma e_4,\\[2mm]
e_4\cdot e_4=- e_2.
 \end{array}
 \]
  \item \underline{The Lie superalgebra $(D^{7}_{1/2 \; 1/2})^3:$} A Novikov structure, depending on a parameter $\lambda \in \mathbb R$, is given by
 \[
 \begin{array}{llll}
e_1\cdot e_1=-\frac{1}{2} e_1 +\lambda e_2, &  e_1\cdot e_2=\frac{1}{2} e_2, &  e_2\cdot e_1= - \frac{1}{2} e_2, &  e_3\cdot e_1=- \frac{1}{2} e_3, \\[2mm] e_3\cdot e_3=\frac{1}{2} e_2,& 
e_4 \cdot e_1=-\frac{1}{2} e_4.
\end{array}
 \]
A Balinsky-Novikov structure, depending on parameters $\lambda, \gamma \in \mathbb R$, is given by
 \[
 \begin{array}{llll}
e_1 \cdot e_1= \gamma e_1+\lambda e_2, & e_1\cdot e_2= (1+\gamma) e_2, & e_1\cdot e_3=\frac{1}{2}(1+\gamma )e_3, & e_1\cdot e_4= \frac{1}{2}(1+\gamma)e_4,\\[2mm]
e_2\cdot e_1=\gamma e_2, & e_3\cdot e_1=\gamma e_3, & e_3\cdot e_3= e_2, & e_4\cdot e_1=\gamma e_4.
 \end{array}
 \]
  \item \underline{The Lie superalgebra $(D^{7}_{1-p,p})$ with $p\leq \frac{1}{2}:$} A Novikov structure, depending on a parameter $\lambda \in \mathbb R$, is given by 
 \[
 \begin{array}{llll}
e_1\cdot e_1=-p e_1 +\lambda e_2, &  e_1\cdot e_2=(1-p) e_2, &  e_1\cdot e_4= (1-2p) e_4, & e_2\cdot e_1=- p e_2, \\[2mm] 
e_3\cdot e_1=-p e_3,& e_3 \cdot e_4=e_2, & e_4\cdot e_1=-p e_4.
\end{array}
 \]
A Balinsky-Novikov structure, depending on parameters $\lambda, \gamma \in \mathbb R$, is given by
 \[
 \begin{array}{llll}
e_1 \cdot e_1= (\gamma-1) e_1+\lambda e_2, & e_1\cdot e_2= \gamma e_2, & e_1\cdot e_3=\frac{1}{2}(2p+\gamma-1 )e_3, & e_1\cdot e_4= \frac{1}{2}(1-2p+\gamma)e_4,\\[2mm]
e_3\cdot e_1=(\gamma -1) e_3, & e_3\cdot e_4= e_2, & e_2\cdot e_1=(\gamma-1)e_2, &  e_4\cdot e_1=(\gamma-1) e_4\\[2mm]
 e_4\cdot e_3=e_2.
 \end{array}
 \]
   \item \underline{The Lie superalgebra $(D^{8}_{1/2}):$} A Novikov structure, depending on a parameter $\lambda\in \mathbb R$, is given by 
 \[
 \begin{array}{llll}
e_1\cdot e_1=-\frac{1}{2} e_1 +\lambda e_2, &  e_1\cdot e_2=\frac{1}{2} e_2, &  e_1\cdot e_4= e_3, & e_2\cdot e_1=- \frac{1}{2} e_2,\\[2mm]
 e_3\cdot e_1=- \frac{1}{2} e_3,& e_4 \cdot e_1= - \frac{1}{2}e_4, & e_4\cdot e_4=\frac{1}{2} e_2.
\end{array}
 \]
A Balinsky-Novikov structure, depending on parameters $\lambda, \gamma \in \mathbb R$, is given by
 \[
 \begin{array}{llll}
e_1 \cdot e_1= \gamma e_1+\lambda e_2, & e_1\cdot e_2= (1+\gamma) e_2, & e_1\cdot e_3=\frac{1}{2}(1+\gamma )e_3, & e_1\cdot e_4= e_3+\frac{1}{2}(1+\gamma)e_4,\\[2mm]
e_2\cdot e_1=\gamma e_2, & e_3\cdot e_1= \gamma e_3, & e_4\cdot e_1=\gamma e_4, &  e_4\cdot e_4= e_2.\\[2mm]
 \end{array}
 \]
    \item \underline{The Lie superalgebra $(D^{9}_{1/2,p})$ with $p>0:$} A Novikov structure, depending on a parameter  $\lambda\in \mathbb R$, is given by 
 \[
 \begin{array}{lllll}
e_1\cdot e_1=-\frac{1}{2} e_1 +\lambda e_2, &  e_1\cdot e_2=\frac{1}{2} e_2, &  e_1\cdot e_3= -p e_4, & e_1\cdot e_4= p e_3, & e_2\cdot e_1=- \frac{1}{2} e_2,\\[2mm]
e_3 \cdot e_1= - \frac{1}{2}e_3, &  e_3\cdot e_3=\frac{1}{2} e_2, & e_4\cdot e_1=-\frac{1}{2} e_4, & e_4\cdot e_4=\frac{1}{2} e_2.
\end{array}
 \]
A Balinsky-Novikov structure, depending on parameters $\lambda, \gamma \in \mathbb R$, is given by
 \[
 \begin{array}{llll}
e_1 \cdot e_1= (\gamma-1) e_1+\lambda e_2, & e_1\cdot e_2= \gamma e_2, & e_1\cdot e_3=\frac{\gamma}{2} e_3-pe_4, & e_1\cdot e_4= p e_3+\frac{\gamma}{2}e_4,\\[2mm]
e_2\cdot e_1=(\gamma-1) e_2, & e_3\cdot e_1= (\gamma-1) e_3, & e_3\cdot e_3= e_2, &  e_4\cdot e_1= (\gamma-1)e_4,\\[2mm]
e_4\cdot e_4=e_2.
 \end{array}
 \]
    \item \underline{The Lie superalgebra $(D^{10}_{0})^1:$} An  LSSA structure, depending on a parameter $\gamma \in \mathbb R$, is given by 
 \[
 \begin{array}{lllll}
e_1\cdot e_1=- e_1, & e_1\cdot e_4=- e_4, & e_2\cdot e_1= - e_2, & e_2\cdot e_4= (1+\gamma) e_3, & e_3\cdot e_1=- e_3,\\[2mm]
e_3 \cdot e_4= - \frac{1}{2}e_2, & e_4\cdot e_1=- e_4, & e_4\cdot e_2=\gamma e_3, & e_4\cdot e_4=\frac{1}{2} e_1.
\end{array}
 \]
{\bf Claim.} There is no Novikov structure on $(D^{10}_{0})^1$ compatible with the Lie structure. 
\begin{proof}
Let us write: 
\[
\begin{array}{lll}
e_i\cdot e_j=\sum_{i,j} \lambda^1_{ij}e_1+\sum_{i,j} \lambda^2_{ij}e_2 & \text{ if } & |e_i|+|e_j|={\bar 0}, \text{ and }\\[2mm]
e_i\cdot e_j=\sum_{i,j} \lambda^3_{ij}e_3+\sum_{i,j} \lambda^4_{ij}e_4 & \text{ if } & |e_i|+|e_j|={\bar 1}.
\end{array}
\]
The compatibility condition implies that
\[
\begin{array}{l}
\lambda^3_{13}=\lambda_{31}^3+1, \quad \lambda^4_{13}=\lambda_{31}^4, \quad \lambda^2_{12}=\lambda_{21}^2+1, \quad \lambda^1_{21}=\lambda_{12}^1, \quad 
\lambda^3_{41}=\lambda_{14}^3, \quad \lambda^4_{41}=\lambda_{14}^4,\\[3mm]
\lambda^3_{24}=\lambda_{42}^3+1, \quad \lambda^4_{42}=\lambda_{24}^4, \quad
\lambda^3_{32}=\lambda_{23}^3, \quad \lambda^4_{32}=\lambda_{23}^4, \quad 
\lambda^1_{33}=0, \quad \lambda^2_{33}=0,\quad 
\lambda^1_{44}=\frac{1}{2}, \\[3mm]
\lambda^2_{44}=0 \quad \lambda^2_{34}=-\lambda^2_{43}- \frac{1}{2}, \quad \lambda^1_{43}=-\lambda^1_{34}.
\end{array}
\]
Let us put
\[
\begin{array}{lcl}
N(x,y,z): & = & (z\cdot x) \cdot y - (-1)^{|x||y|} (z\cdot y) \cdot x,\\[3mm]
T(x,y,z) : & = & (x,y,z)- (-1)^{|x||y|} (y,x,z).
\end{array}
\]
A direct calculation gives
\[
N(e_4, e_4, e_4)=0 \implies \lambda^3_{14}=\lambda_{14}^4=0, \text{ also } N(e_4, e_1, e_4)=0 \implies \lambda^1_{11}=\lambda_{11}^2=0.
\]
Similarly, 
\[
N(e_1, e_3, e_1)=0 \implies \lambda^3_{31}=\lambda_{13}^4=0, \text{ also } T(e_2, e_1, e_2)=0 \implies \lambda^2_{21}=\lambda_{12}^1=\lambda^1_{22}=\lambda_{22}^2=0.
\]
Additionally, 
\[
N(e_4, e_2, e_1)=0 \implies \lambda^4_{24}=\lambda_{42}^3=-1 \text{ also } N(e_4, e_2, e_4)=0 \implies \lambda^1_{34}=\lambda_{43}^2=0.
\]
Now, $N(e_3, e_4, e_4)=\frac{1}{2}e_3$ which is a not zero. Therefore, it cannot be Novikov.   
\end{proof}
However, there is a Balinsky-Novikov structure, depending on parameters $\lambda, \gamma \in \mathbb R$, is given by
 \[
 \begin{array}{llll}
e_1 \cdot e_1= \gamma e_1+\lambda e_2, & e_1\cdot e_2= (1+\gamma) e_2, & e_1\cdot e_3=\frac{1}{2} (2+\gamma)e_3, & e_1\cdot e_4= -\lambda e_3+\frac{\gamma}{2}e_4,\\[2mm]
e_2\cdot e_1=\gamma e_2, & e_2\cdot e_4= -\gamma e_3, & e_3\cdot e_1= \gamma e_3, &  e_3\cdot e_4= -\frac{1}{2} e_2,\\[2mm]
e_4\cdot e_1=-2 \lambda e_3+\gamma e_4, & e_4\cdot e_2=-2(1+\gamma)e_3& e_4\cdot e_3=-\frac{1}{2}e_2, & e_4\cdot e_4=e_1.
 \end{array}
 \]
    \item \underline{The Lie superalgebra $(D^{10}_{0})^2:$} An LSSA structure, depending on parameters $\lambda, \gamma \in \mathbb R$, is given by
    \[
\begin{array}{lllll}
e_1\cdot e_1=- e_1 -2 \lambda \gamma e_2, &  e_1\cdot e_4= \gamma e_3-e_4, & e_2\cdot e_1= - e_2, & e_2\cdot e_4= e_3, & e_3\cdot e_1=- e_3,\\[2mm]
e_3 \cdot e_4= \frac{1}{2} (1-2 \lambda)e_2, & e_4\cdot e_1= \gamma e_3- e_4, & e_4\cdot e_3=\lambda e_2, & e_4\cdot e_4=-\frac{1}{2} e_1.
\end{array}
 \]
{\bf Claim.} There is no Novikov structure on $(D^{10}_{0})^2$ compatible with the Lie structure. 

However, there is a Balinsky-Novikov structure, depending on parameters $\lambda, \gamma \in \mathbb R$, is given by given by 
 \[
 \begin{array}{lll}
e_1 \cdot e_1= (\gamma-1) e_1+\lambda e_2, & e_1\cdot e_2= \gamma e_2, & e_1\cdot e_3=\frac{1}{2} (1+\gamma)e_3, \\[2mm] e_1\cdot e_4= -\lambda e_3+\frac{1}{2}(\gamma-1)e_4,&
e_2\cdot e_1=(\gamma-1) e_2, & e_2\cdot e_4= (1-\gamma) e_3, \\[2mm]
e_3\cdot e_1= (\gamma-1) e_3, &  e_3\cdot e_4= \frac{1}{2} e_2,&
e_4\cdot e_1=-2 \lambda e_3+(\gamma-1)e_4, \\[2mm] 
e_4\cdot e_2=-2 \gamma e_3& e_4\cdot e_3=\frac{1}{2}e_2, & e_4\cdot e_4=-e_1.
 \end{array}
 \]
   \item \underline{The Lie superalgebra $(2A_{1,1}+2A)^1:$} A Novikov structure, depending on parameters $\lambda, \gamma \in \mathbb R$, is given by
    \[
\begin{array}{l}
e_3\cdot e_3= \frac{1}{2}e_1 , \quad e_3\cdot e_4= \lambda e_1 +\gamma e_2, \quad e_4\cdot e_3= - \lambda e_1 -\gamma e_2 ,  \quad e_4\cdot e_4= \frac{1}{2} e_2.
\end{array}
 \]
A Balinsky-Novikov structure, depending on parameters $\lambda, \gamma \in \mathbb R$, is given by 
 \[
 \begin{array}{llll}
e_1 \cdot e_1= \gamma e_1, & e_1\cdot e_3= \frac{\gamma}{2} e_3, & e_2\cdot e_2=2 \lambda e_2, & e_2\cdot e_4= \lambda e_4,\\[2mm]
e_3\cdot e_1=\gamma e_3, & e_3\cdot e_3=  e_1, & e_4\cdot e_2= 2 \lambda  e_4, &  e_4\cdot e_4= e_2.
 \end{array}
 \]
  \item \underline{The Lie superalgebra $(2A_{1,1}+2A)^2:$} A Novikov structure, depending on a parameter $\gamma \in \mathbb R$, is given by 
    \[
\begin{array}{l}
e_3\cdot e_3= \frac{1}{2}e_1 , \quad e_3\cdot e_4= \gamma e_1, \quad e_4\cdot e_3= (1- \gamma) e_1 ,  \quad e_4\cdot e_4= \frac{1}{2} e_2.
\end{array}
 \]
 A Balinsky-Novikov structure, depending on parameters $\lambda, \gamma \in \mathbb R$, is given by
 \[
 \begin{array}{llll}
e_1 \cdot e_1= 2\gamma e_1, & e_1\cdot e_2= 2 \gamma e_1, & e_1\cdot e_3= \gamma e_3, &e_1\cdot e_4=\gamma e_3, \\[2mm]
e_2\cdot e_1= 2 \gamma e_1,& 
e_2\cdot e_2=-2(\lambda-\gamma) e_1+2 \lambda e_2, & e_2\cdot e_3=  \gamma e_3, & 
e_2\cdot e_4= (\gamma - \lambda)  e_3+\lambda e_4, \\[2mm]  e_3\cdot e_1= 2 \gamma e_3,& 
e_3 \cdot e_2=2 \gamma e_3, & e_3 \cdot e_3=e_1, & 
e_3 \cdot e_4=e_1, \\[2mm]
e_4\cdot e_1=2 \gamma e_3, & e_4\cdot e_2=-2(\lambda- \gamma) e_3+ 2 \lambda e_4,& e_4\cdot e_3=e_1, & e_4 \cdot e_4=e_2.
 \end{array}
 \]
   \item \underline{The Lie superalgebra $(2A_{1,1}+2A)^3_p:$} A Novikov structure, depending on a parameter $\gamma \in \mathbb R$, is given by
    \[
\begin{array}{l}
e_3\cdot e_3= \frac{1}{2}e_1 , \quad e_3\cdot e_4= \gamma e_1, \quad e_4\cdot e_3= (p- \gamma) e_1 + p e_2,  \quad e_4\cdot e_4= \frac{1}{2} e_2.
\end{array}
\]
A Balinsky-Novikov structure, depending on parameters $\lambda, \gamma \in \mathbb R$, is given by
 \[
 \begin{array}{lll}
e_1 \cdot e_1= \frac{2}{p} (p^2 \gamma + \lambda(2 p^2-1)) e_1 + 2p \lambda e_2, & e_1\cdot e_2= -2 p\lambda e_1+2 p \gamma e_2, \\[2mm] 
e_1\cdot e_3= \frac{1}{p}(p^2\gamma +\lambda(p^2-1)) e_3+ \lambda e_4, &  
e_1\cdot e_4=-\lambda e_3+p( \gamma+\lambda)e_4,\\[2mm] 
e_2\cdot e_1= -2p \lambda e_1+2p \gamma e_2,& 
e_2\cdot e_2=-2 p \gamma e_1- \frac{2}{p}(p^2 \lambda+(2p^2-1)\gamma) e_2, 
\\[2mm] 
e_2\cdot e_3=  -p( \lambda +\gamma)  e_3 + \gamma e_4, & 
e_2\cdot e_4= -\gamma  e_3 + \frac{1}{p}((1-p^2)\gamma-p^2\lambda) e_4, \\[2mm]  e_3\cdot e_1= \frac{2}{p}(p^2\gamma +(p^2-1)\lambda) e_3 +2 \lambda e_4, & 
e_3 \cdot e_2=-2p ( \gamma + \lambda) e_3 + 2 \gamma e_4, \\[2mm] e_3 \cdot e_3=e_1, &
e_3 \cdot e_4=p e_1+p e_2, \\[2mm]
e_4\cdot e_1=-2 \lambda e_3+ 2p(\gamma + \lambda)e_4, & 
e_4\cdot e_2=-2 \gamma e_3  - \frac{2}{p}((p^2-1)\gamma+p^2\lambda)) e_4, \\[2mm] e_4\cdot e_3=p e_1 + p e_2, & e_4 \cdot e_4=e_2.
 \end{array}
 \]
   \item \underline{The Lie superalgebra $(2A_{1,1}+2A)^4_p:$} A Novikov structure, depending on a parameter $\gamma \in \mathbb R$, is given by
    \[
\begin{array}{l}
e_3\cdot e_3= \frac{1}{2}e_1 , \quad e_3\cdot e_4= \gamma e_1, \quad e_4\cdot e_3=(p-\gamma) e_1-pe_2, \quad e_4\cdot e_4= \frac{1}{2} e_2.
\end{array}
 \]
A Balinsky-Novikov structure, depending on parameters $\lambda, \gamma \in \mathbb R$, is given by
 \[
 \begin{array}{lll}
e_1 \cdot e_1= -\frac{2}{p} (p^2 \gamma - \lambda(2 p^2+1)) e_1 - 2p \lambda e_2, & e_1\cdot e_2= 2 p\lambda e_1-2 p \gamma e_2, \\[2mm] 
e_1\cdot e_3= \frac{1}{p}(-p^2\gamma +\lambda(p^2+1)) e_3+ \lambda e_4, &  
e_1\cdot e_4=\lambda e_3+p( -\gamma+\lambda)e_4,\\[2mm] 
e_2\cdot e_1= 2p \lambda e_1-2p \gamma e_2,& 
e_2\cdot e_2= 2 p \gamma e_1- \frac{2}{p}(-p^2 \lambda+(2p^2+1)\gamma) e_2, 
\\[2mm] 
e_2\cdot e_3=  p( -\gamma +\lambda)  e_3 + \gamma e_4, & 
e_2\cdot e_4= \gamma  e_3 + \frac{1}{p}((-1-p^2)\gamma+p^2\lambda) e_4, \\[2mm]  e_3\cdot e_1= \frac{2}{p}(-p^2\gamma +(p^2+1)\lambda) e_3 +2 \lambda e_4, & 
e_3 \cdot e_2=-2p ( \gamma - \lambda) e_3 + 2 \gamma e_4, \\[2mm] e_3 \cdot e_3=e_1, &
e_3 \cdot e_4=p e_1-p e_2, \\[2mm]
e_4\cdot e_1=2 \lambda e_3- 2p(\gamma - \lambda)e_4, & 
e_4\cdot e_2=2 \gamma e_3  - \frac{2}{p}((p^2+1)\gamma-p^2\lambda)) e_4, \\[2mm] e_4\cdot e_3=p e_1 - p e_2, & e_4 \cdot e_4=e_2.
 \end{array}
 \]
    \item \underline{The Lie superalgebra $(C_1^1+A):$} A Novikov structure, depending on parameters $\lambda, \gamma \in \mathbb R$, is given by
    \[
\begin{array}{llll}
e_1\cdot e_1= \lambda e_1 +\gamma e_2 , & e_1\cdot e_2= (1+\lambda) e_2, & e_1\cdot e_3= (1+\lambda) e_3, & e_1\cdot e_4=\lambda  e_4,\\[2mm] e_2\cdot e_1=\lambda e_2, & 
e_3\cdot e_1= \lambda e_3, & e_3 \cdot e_4=-\lambda e_2, & e_4\cdot e_1=\lambda e_4,\\[2mm]  
e_4 \cdot e_3=(1+\lambda) e_2.
\end{array}
 \]
A Balinsky-Novikov structure, depending on parameters $\lambda, \gamma \in \mathbb R$, is given by
 \[
 \begin{array}{llll}
e_1 \cdot e_1= (\gamma -1) e_1 + \lambda e_2, & e_1 \cdot e_2= \gamma e_2, & 
e_1 \cdot e_3=\frac{1}{2}(1+\gamma) e_3,&
e_1 \cdot e_4=\frac{1}{2}(\gamma-1)e_4,\\[2mm]
e_2\cdot e_1= (\gamma-1) e_2, & e_3\cdot e_1=(\gamma-1)e_3, & e_3\cdot e_4= e_2, & e_4\cdot e_1=(\gamma-1)e_4, \\[2mm] e_4\cdot e_3=e_2.
 \end{array}
 \]
   \item \underline{The Lie superalgebra $(C_{1/2}^1+A):$} A Novikov structure is given by 
    \[
\begin{array}{l}
e_1\cdot e_1= -\frac{1}{2} e_1 +\gamma e_2 , \quad e_1\cdot e_2= \frac{1}{2} e_2, \quad e_2\cdot e_1= -\frac{1}{2} e_2, \quad e_3\cdot e_1=-\frac{1}{2} e_3,\quad  e_3\cdot e_3=\frac{1}{2}e_2.
\end{array}
 \]
A Balinsky-Novikov structure, depending on parameters $\lambda, \gamma \in \mathbb R$, is given by
 \[
 \begin{array}{llll}
e_1 \cdot e_1= \gamma e_1 + \lambda e_2, & e_1 \cdot e_2= (1+\gamma) e_2, & 
e_1 \cdot e_3=\frac{1}{2}(1+\gamma) e_3,&
e_1 \cdot e_4=\frac{\gamma}{2}e_4,\\[2mm]
e_2\cdot e_1= \gamma e_2, & e_3\cdot e_1=\gamma e_3, & e_3\cdot e_3= e_2, & e_4\cdot e_1=\gamma e_4.
 \end{array}
 \]
  \item \underline{The Lie superalgebra $(C_{-1}^2+A):$} A Novikov structure, depending on parameters $\lambda, \gamma \in \mathbb R$, is given by
    \[
\begin{array}{l}
e_1\cdot e_1= (2\lambda-1) e_1 +\gamma e_2 , \quad e_1\cdot e_2= (2\lambda-1) e_2, \quad e_1\cdot e_3= 2 \lambda  e_3, \quad e_1\cdot e_4=2(\lambda-1) e_4,\\[2mm]
e_2\cdot e_1=(2\lambda -1)e_2, \quad e_3\cdot e_1=(2 \lambda -1) e_3, \quad e_3 \cdot e_4=(1-\lambda) e_2, \quad e_4\cdot e_1=(2\lambda -1) e_4, \\[2mm] 
e_4\cdot e_3=\lambda e_2.
\end{array}
 \]
A Balinsky-Novikov structure, depending on parameters $\lambda, \gamma \in \mathbb R$, is given by
 \[
 \begin{array}{llll}
e_1 \cdot e_1= \gamma e_1 + \lambda e_2, & e_1 \cdot e_2= \gamma e_2, & 
e_1 \cdot e_3=\frac{1}{2}(2+\gamma) e_3,&
e_1 \cdot e_4=\frac{1}{2} (\gamma -2)e_4,\\[2mm]
e_2\cdot e_1= \gamma e_2, & e_3\cdot e_1=\gamma e_3, & e_3\cdot e_4= e_2, & e_4\cdot e_1=\gamma e_4,\\[2mm]
e_4\cdot e_3=e_2.
 \end{array}
 \]
 \item \underline{The Lie superalgebra $(C^3+A):$} A Novikov structure is given by
    \[
    \begin{array}{c}
e_4 \cdot e_1=- e_3, \quad e_4\cdot e_4=\frac{1}{2}e_2.
\end{array}
    \]
A Balinsky-Novikov structure, depending on parameters $\lambda, \gamma \in \mathbb R$, is given by
 \[
 \begin{array}{llll}
e_1 \cdot e_1= \gamma e_1 + \lambda e_2, & e_1 \cdot e_2= \gamma e_2, & 
e_1 \cdot e_3=\frac{\gamma}{2}e_3,&
e_1 \cdot e_4=e_3+\frac{\gamma}{2} e_4,\\[2mm]
e_2\cdot e_1= \gamma e_2, & e_3\cdot e_1=\gamma e_3, & e_4\cdot e_1=\gamma e_4, & e_4\cdot e_4=e_2.
 \end{array}
 \]
      \item \underline{The Lie superalgebra $(C_{0}^5+A):$} A  Novikov structure, depending on parameters $\lambda, \gamma \in \mathbb R$, is given by
    \[
\begin{array}{l}
e_1\cdot e_1= 2\lambda e_1 +\gamma e_2 , \quad e_1\cdot e_2= 2\lambda e_2, \quad e_1\cdot e_3= 2 \lambda  e_3 -e_4, \quad e_1\cdot e_4=e_3+ 2\lambda e_4,\\[2mm]
e_2\cdot e_1=2\lambda e_2, \quad e_3\cdot e_1=2 \lambda e_3, \quad e_3 \cdot e_3=\frac{1}{2} e_2, \quad e_3\cdot e_4=-\lambda e_2, \quad e_4\cdot e_1=2 \lambda e_4, \\[2mm] 
e_4\cdot e_3=\lambda e_2, \quad e_4\cdot e_4=\frac{1}{2}e_2.
\end{array}
 \]
A Balinsky-Novikov structure, depending on parameters $\lambda, \gamma \in \mathbb R$, is given by
 \[
 \begin{array}{llll}
e_1 \cdot e_1= \gamma e_1 + \lambda e_2, & e_1 \cdot e_2= \gamma e_2, & 
e_1 \cdot e_3=\frac{\gamma}{2}e_3-e_4,&
e_1 \cdot e_4=e_3+\frac{\gamma}{2} e_4,\\[2mm]
e_2\cdot e_1= \gamma e_2, & e_3\cdot e_1=\gamma e_3, & e_3\cdot e_3= e_2, & e_4\cdot e_1=\gamma e_4,\\[2mm]
e_4\cdot e_4=e_2.
 \end{array}
 \]
    \item \underline{The Lie superalgebra $D^1$:} A Novikov structure, depending on parameters $\lambda,\mu,  \gamma \in \mathbb R$, is given by
    \[
\begin{array}{l}
e_2\cdot e_2= \lambda e_1 +\gamma e_3 , \quad e_2\cdot e_3= (1+\mu) e_1, \quad e_2\cdot e_4=e_4, \quad e_3\cdot e_2= \mu e_1.
\end{array}
 \]
A Balinsky-Novikov structure, depending on parameters $\lambda, \gamma \in \mathbb R$, is given by
 \[
 \begin{array}{llll}
e_2 \cdot e_2= \gamma e_1 + \lambda e_3, & e_2 \cdot e_4=  e_4, & 
e_3 \cdot e_2=-e_1.
 \end{array}
 \]
    \item \underline{The Lie superalgebra $D^2_q$ with $q\not = 0$:} A Novikov structure is given by  
    \[
\begin{array}{l}
e_1\cdot e_3= -q e_1 , \quad e_2\cdot e_3= -q  e_2, \quad e_3\cdot e_1=-(1+q)e_1, \quad e_3\cdot e_2= -e_1- (q+1)e_2, \\[2mm]
e_3\cdot e_3=-q e_3, \quad e_4\cdot e_3=-q e_4.
\end{array}
 \]
A Balinsky-Novikov structure, depending on parameters $\lambda, \gamma \in \mathbb R$, is given by
 \[
 \begin{array}{llll}
e_3 \cdot e_1= - e_1, & e_3 \cdot e_2=  -e_1-e_2, & 
e_3 \cdot e_3= \lambda e_1+\gamma e_2,& e_3\cdot e_4=q e_4.
 \end{array}
 \]
   \item \underline{The Lie superalgebra $D^3_{pq}$ with $pq\not = 0$:} A  Novikov structure, depending on parameters $\lambda, \mu,  \gamma \in \mathbb R$, is given by
    \[
\begin{array}{l}
e_1\cdot e_3= (p+\gamma) e_1 , \quad e_2\cdot e_3= (p+\gamma)  e_2, \quad e_3\cdot e_1= \gamma e_1+e_2, \quad e_3\cdot e_2= -e_1+ \gamma e_2, \\[2mm]
e_3\cdot e_3= \lambda e_1 +\mu e_2 + (p+\gamma)e_3, \quad e_3\cdot e_4=(p+q+\gamma) e_4, \quad e_4 \cdot e_3=(p+\gamma)e_4.
\end{array}
 \]
A Balinsky-Novikov structure, depending on parameters $\lambda, \gamma \in \mathbb R$, is given by
 \[
 \begin{array}{llll}
e_3 \cdot e_1= - p e_1 + e_2, & e_3 \cdot e_2=  -e_1- p e_2, & 
e_3 \cdot e_3= \lambda e_1+\gamma e_2,& e_3\cdot e_4=q e_4.
 \end{array}
 \]
  \item \underline{The Lie superalgebra $D^{11}_{pq}$ with $0<|p|\leq |q|\leq 1$:} A Novikov structure is given by 
    \[
\begin{array}{l}
e_1\cdot e_2=  e_2 , \quad e_1\cdot e_3= p  e_3, \quad e_1\cdot e_4= q e_4.
\end{array}
 \]
A Balinsky-Novikov structure is given by
 \[
 \begin{array}{llll}
e_1 \cdot e_2=  e_2, & e_1 \cdot e_3=p e_3, & 
e_1 \cdot e_4= q e_4. 
 \end{array}
 \]
    \item \underline{The Lie superalgebra $D^{12}$:} A Novikov structure, depending on  a parameter $ \gamma \in \mathbb R$, is given by
    \[
\begin{array}{l}
e_1\cdot e_2=e_2, \quad e_1\cdot e_4= (1+\gamma) e_3 , \quad e_4\cdot e_1= \gamma  e_3. 
\end{array}
 \]
A Balinsky-Novikov structure is given by
 \[
 \begin{array}{llll}
e_1 \cdot e_1=  -2e_1, & e_1 \cdot e_3
=- e_3, & 
e_1 \cdot e_4= e_3- e_4,& e_2\cdot e_1=-2e_2,\\[2mm]
e_3 \cdot e_1=-2 e_3, & e_4\cdot e_1=-2e_4.
 \end{array}
 \]
  \item \underline{The Lie superalgebra $D^{13}_p$ with  $p\not =0$:} A Novikov structure is given by
    \[
\begin{array}{l}
e_1\cdot e_2= p e_2 , \quad e_1\cdot e_3=  e_3, \quad e_1\cdot e_4=e_3+e_4. 
\end{array}
 \]
A Balinsky-Novikov structure, depending on a parameter  $\gamma \in \mathbb R$, is given by
 \[
 \begin{array}{llll}
e_1 \cdot e_1=  2(\gamma -1) e_1, & e_1 \cdot e_2
= (p-1+\gamma)e_2, & 
e_1 \cdot e_3= \gamma e_3,& e_1\cdot e_4=e_3+ \gamma e_4,\\[2mm]
e_2 \cdot e_1=2(\gamma-1) e_2, & e_3\cdot e_1=2 (\gamma -1) e_3,& e_4\cdot e_1=2(\gamma-1)e_4.
 \end{array}
 \]
  \item \underline{The Lie superalgebra $D^{14}_{pq}$ with $p\not =0, q\geq 0$:} A Novikov structure is given by 
    \[
\begin{array}{l}
e_1\cdot e_2= p e_2 , \quad e_1\cdot e_3=  qe_3-e_4, \quad e_1\cdot e_4=e_3+qe_4. 
\end{array}
 \]
A Balinsky-Novikov structure is given by
 \[
 \begin{array}{llll}
e_1 \cdot e_1=  -2 q e_1, & e_1 \cdot e_2
= (p-q)e_2, & 
e_1 \cdot e_3= - e_4,& e_1\cdot e_4=e_3,\\[2mm]
e_2 \cdot e_1=-2q e_2, & e_3\cdot e_1=- 2 q e_3,& e_4\cdot e_1=-2 q e_4.
 \end{array}
 \]
  \item \underline{The Lie superalgebra $D^{15}$:} A Novikov structure, depending on a parameter $\gamma \in \mathbb R$, is given by
    \[
\begin{array}{l}
e_1\cdot e_3=e_2, \quad e_1\cdot e_4= \gamma e_2 + e_3 , \quad e_4\cdot e_1= \gamma e_2.
\end{array}
 \]
A Balinsky-Novikov structure, depending on a parameter $\gamma \in \mathbb R$, is given by
 \[
 \begin{array}{llll}
e_1 \cdot e_1=  2 \gamma  e_1, & e_1 \cdot e_2
= \gamma e_2, & 
e_1 \cdot e_3= e_2+ \gamma e_3,& e_1\cdot e_4=e_3 + \gamma e_4,\\[2mm]
e_2 \cdot e_1= 2 \gamma  e_2, & e_3\cdot e_1= 2 \gamma  e_3,& e_4\cdot e_1=2 \gamma  e_4.
 \end{array}
 \]
  \item \underline{The Lie superalgebra $D^{16}$:} A Novikov structure is given by 
    \[
\begin{array}{l}
e_1\cdot e_2= e_2, \quad e_1\cdot e_3= e_2+e_3, \quad e_1\cdot e_4=e_3+e_4.
\end{array}
 \]
A Balinsky-Novikov structure, depending on a parameter $\gamma \in \mathbb R$, is given by
 \[
 \begin{array}{llll}
e_1 \cdot e_1=  2 (\gamma-1)  e_1, & e_1 \cdot e_2
= \gamma e_2, & 
e_1 \cdot e_3= e_2+ \gamma e_3,& e_1\cdot e_4=e_3 + \gamma e_4,\\[2mm]
e_2 \cdot e_1= 2 (\gamma-1)  e_2, & e_3\cdot e_1= 2 (\gamma-1)  e_3,& e_4\cdot e_1=2 (\gamma-1)  e_4.
 \end{array}
 \]
   \item \underline{The Lie superalgebra $(A_{3,1}+A)$:} A Novikov structure, depending on parameters $\lambda, \mu, \gamma \in \mathbb R$, is given by  
    \[
\begin{array}{l}
e_2\cdot e_2= \gamma e_1, \quad e_2\cdot e_3= (1+\lambda) e_1, \quad e_3\cdot e_2= \lambda e_1, \quad e_3\cdot e_3=\mu e_1, \quad e_4\cdot e_4=\frac{1}{2} e_1.
\end{array}
 \]
A Balinsky-Novikov structure, depending on parameters $\lambda, \mu, \gamma \in \mathbb R$, is given by 
 \[
 \begin{array}{lllll}
e_2 \cdot e_2=  \gamma  e_1, & e_2 \cdot e_3
= \lambda e_1, & 
e_3 \cdot e_2= (\lambda -1)e_1,& e_3\cdot e_3=\mu e_1,& e_4\cdot e_4=e_1.
 \end{array}
 \]
  \item \underline{The Lie superalgebra $(D^3_{p,-1/2})$ with $p\not =0$:} A  Novikov structure, depending on parameters $\lambda, \gamma \in \mathbb R$, is given by 
    \[
\begin{array}{l}
e_1\cdot e_1=-\frac{1}{2}e_1+ \lambda e_2+\gamma e_3, \quad e_1\cdot e_2= \frac{1}{2} e_2, \quad e_1\cdot e_3= \frac{1}{2}(2p-1) e_3, \quad e_2\cdot e_1=-\frac{1}{2} e_2,\\[2mm]  
e_3\cdot e_1=-\frac{1}{2} e_3, \quad e_4\cdot e_1=-\frac{1}{2}e_4, \quad e_4\cdot e_4=\frac{1}{2}e_2.
\end{array}
\]
A Balinsky-Novikov structure, depending on parameters $\lambda, \gamma \in \mathbb R$, is given by 
\[
 \begin{array}{lllll}
e_1 \cdot e_1=  \gamma  e_2 + \lambda e_3, & e_1 \cdot e_2
= e_2, & 
e_1 \cdot e_3= p e_3, & e_1\cdot e_4=\frac{1}{2} e_4, & e_4\cdot e_4=e_2.
 \end{array}
 \]
  \item \underline{The Lie superalgebra 
  $(D^2_{-1/2})^1$:} A Novikov structure, depending on parameters $\lambda, \gamma \in \mathbb R$, is given by 
    \[
\begin{array}{l}
e_1\cdot e_1=-\frac{1}{2}e_1+ \lambda e_2+\gamma e_3, \quad e_1\cdot e_2= \frac{1}{2} e_2, \quad e_1\cdot e_3= e_2+ \frac{1}{2} e_3, \quad e_2\cdot e_1=-\frac{1}{2} e_2, \\[2mm] 
e_3\cdot e_1=-\frac{1}{2} e_3, \quad 
e_4\cdot e_1= -\frac{1}{2} e_4,\quad e_4\cdot e_4=\frac{1}{2}e_2.
\end{array}
 \]
A Balinsky-Novikov structure, depending on parameters $\lambda, \gamma \in \mathbb R$, is given by 
 \[
 \begin{array}{lllll}
e_1 \cdot e_1=  \gamma  e_2+\lambda e_3, & e_1 \cdot e_2
= e_2, & 
e_1 \cdot e_3= e_2+e_3,& e_1\cdot e_4=\frac{1}{2} e_4,& e_4\cdot e_4=e_2.
 \end{array}
 \]

  \item \underline{The Lie superalgebra 
  $(D^2_{-1/2})^2$:} A Novikov structure, depending on parameters $\lambda, \gamma \in \mathbb R$, is given by
    \[
\begin{array}{l}
e_1\cdot e_1=-\frac{1}{2}e_1+ \lambda e_2+\gamma e_3, \quad e_1\cdot e_2= \frac{1}{2} e_2, \quad e_1\cdot e_3= -e_2+ \frac{1}{2} e_3, \quad e_2\cdot e_1=-\frac{1}{2} e_2, \\[2mm] 
e_3\cdot e_1=-\frac{1}{2} e_3, \quad 
e_4\cdot e_1= -\frac{1}{2} e_4,\quad e_4\cdot e_4=\frac{1}{2}e_2.
\end{array}
 \]
A Balinsky-Novikov structure, depending on parameters $\lambda, \gamma \in \mathbb R$, is given by
 \[
 \begin{array}{lllll}
e_1 \cdot e_1=  \gamma  e_2+\lambda e_3, & e_1 \cdot e_2
= e_2, & 
e_1 \cdot e_3= -e_2+e_3,& e_1\cdot e_4=\frac{1}{2} e_4,& e_4\cdot e_4=e_2.
 \end{array}
 \]
  \item \underline{The Lie superalgebra 
  $(A_{1,1}+3A)^1$:} A Novikov structure, depending on parameters $\lambda, \mu, \gamma \in \mathbb R$, is given by
    \[
\begin{array}{l}
e_2\cdot e_2=\frac{1}{2}e_1, \quad e_2\cdot e_3= \lambda e_1, \quad e_2\cdot e_4= \gamma e_1, \quad e_3\cdot e_2=-\lambda e_1, \quad e_3\cdot e_3=\frac{1}{2} e_1,\\[2mm]
e_3\cdot e_4= \mu e_1,\quad e_4\cdot e_2=-\gamma e_1, \quad e_4\cdot e_3=-\mu  e_1, \quad e_4\cdot e_4=\frac{1}{2}e_1.
\end{array}
 \]
A Balinsky-Novikov structure is given by:
 \[
 \begin{array}{llll}
e_2 \cdot e_2=   e_1, & e_3 \cdot e_3
= e_1, & 
e_4 \cdot e_4= e_1.
 \end{array}
 \]

 \item \underline{The Lie superalgebra 
  $(A_{1,1}+3A)^2$:} A Novikov structure, depending on parameters $\lambda, \mu, \gamma \in \mathbb R$, is given by
    \[
\begin{array}{l}
e_2\cdot e_2=\frac{1}{2}e_1, \quad e_2\cdot e_3= \lambda e_1, \quad e_2\cdot e_4= \gamma e_1, \quad e_3\cdot e_2=-\lambda e_1, \quad e_3\cdot e_3=\frac{1}{2} e_1,\\[2mm]
e_3\cdot e_4= \mu e_1,\quad e_4\cdot e_2=-\gamma e_1, \quad e_4\cdot e_3=-\mu  e_1, \quad e_4\cdot e_4=-\frac{1}{2}e_1.
\end{array}
 \]
A Balinsky-Novikov structure is given by
 \[
 \begin{array}{llll}
e_2 \cdot e_2=   e_1, & e_3 \cdot e_3
= e_1, & 
e_4 \cdot e_4= -e_1.
 \end{array}
 \]
 \end{enumerate}
 We conclude by saying that the lists of Lie superalgebras classified by Backhouse are all left-symmetric. Moreover, they are all Novikov and Balinsky-Novikov superalgebras, except the two $(D_0^{10})^1$  and $(D_0^{10})^2$ that are Balinsky-Novikov but not Novikov. 
\section{The list of four-dimensional real Lie superalgebras}
The tables below differ slightly from those in \cite{BM}. The symplectic forms given in \cite{BM} are given in terms of parameters. For convenience, we specify these  parameters here. In addition, we correct the statement regarding the Lie superalgebras $(D_0^{10})^1$ and $(D_0^{10})^2$. The two admit both an odd and an even non-degenerate closed form, contrary to what is stated in \cite{BM}.
\footnotesize
\begin{table}[H]
\centering
\begin{tabular}{| c | l | l | c| }\hline
 The LSA & Relations in the basis: $e_1, e_2 \, |\, e_3, e_4$ & Symplectic structure $\omega$  & $|\omega|$ \\ \hline
 $D^5$ & 
$
\begin{array}{lcllcl}
[e_1,e_3]&=&e_3,&
[e_1, e_4]&=&e_4, \\[1mm]
[e_2, e_4]&=&e_3
\end{array}
$

& $\begin{array}{l} e_3^*\wedge e_1^*+ e_4^*\wedge e_2^*
\end{array}$  & $\bar 1$  \\[2mm]  \hline
$D^6$ &  
$
\begin{array}{lcrlcl}
[e_1, e_3]&=&e_3, \;  [e_1, e_4]&=&e_4,\\[1mm]
[e_2, e_3]&=&-e_4, \; 
 [e_2, e_4]&=&e_3
\end{array}
$& $ \begin{array}{l}
e_3^*\wedge e_1^* + e_4^*\wedge e_2^*
\end{array}
$
& $\bar 1$  \\[2mm] \hline
$\begin{array}{c}
D^7_{pq},\\[1mm]
 pq\not =0,\\[1mm]
p\geq q
\end{array}
$ 
& 
$
\begin{array}{lcllcl}
[e_1,e_2]&=&e_2, &  
 [e_1, e_3]&=&p e_3,\\[1mm]
[e_1, e_4]&=&q e_4
\end{array}$ & \text{None} &--  \\[2mm] \hline
$
\begin{array}{c}
D^7_{-1q},\\[1mm]
q\leq -1 
\end{array} 
$
&
$
\begin{array}{lcllcl}
[e_1,e_2]&=&e_2,& 
[e_1, e_3]&=&- e_3,\\[1mm]
 [e_1, e_4]&=&q e_4
\end{array}$ & $ \begin{array}{l} e_4^*\wedge e_1^*+ e_3^*\wedge e_2^* \end{array} $ &
${\bar 1}$  \\[2mm] \hline
$\begin{array}{c}
D^7_{pp},\\[2mm]
p=-q\\[2mm]
q\not =0, -1
\end{array}
$
& 
$
\begin{array}{lcllcl}
[e_1, e_2] & = & e_2, &  
[e_1, e_3] & = & p e_3\\[1mm]
[e_1, e_4] & = & -p e_4
\end{array}
$
&
$\begin{array}{l}
e_2^*\wedge e_1^*-e_3^*\wedge e_4^*
\end{array}$
&
$\bar 0$
\\[2mm] \hline
$
\begin{array}{c}
D^7_{pq},\\[1mm]
p=-q , \\[1mm]
q =-1
\end{array}
$
&
$
\begin{array}{lcllcl}
[e_1,e_2]&=&e_2,&  
 [e_1, e_3]&=& e_3,\\[1mm]
 [e_1, e_4]&=&- e_4
\end{array}$ & $
\begin{array}{l}
e_2^*\wedge e_1^*- e_3^*\wedge e_4^*\\[1mm]
e_3^*\wedge e_1^*+e_4^*\wedge e_2^*
\end{array}$ & 
 
$\begin{array}{l}
\bar 0\\[1mm]
\bar 1
\end{array}$   \\[2mm] \hline
$
\begin{array}{c}
D^8_{p},\\
p\not=0 
\end{array} 
$
&
$
\begin{array}{lcllcl}
[e_1, e_2]&=&e_2,&  
[e_1, e_3]&=&p e_3,\\[1mm]
 [e_1, e_4]&=&e_3+ pe_4
\end{array}$ & None & --
\\[2mm] \hline
$D^8_{-1}$ & 
$
\begin{array}{lcllcl}
[e_1, e_2]&=&e_2,& 
 [e_1, e_3]&=&- e_3, \\[1mm]
[e_1, e_4]&=&e_3-e_4
\end{array}$ & $ \begin{array}{l}
e_3^*\wedge e_1^*+ e_4^*\wedge e_2^*
\end{array}$ & $\bar 1$\\[2mm] \hline
$
\begin{array}{c}
D^9_{pq}\\[1mm]
q>0
\end{array} 
$

& 
$
\begin{array}{lcllcl}
 [e_1, e_3]&=&pe_3-qe_4, &  
 [e_1, e_2]&=&e_2, \\[2mm]
 [e_1, e_4]&=&q e_3+pe_4
\end{array}$ & None  &--\\[2mm] \hline
  $D^{10}_{q}$ & $
\begin{array}{lcllcl}
[e_1, e_2]&=&e_2, &  
 [e_1, e_3]&=&(q+1)e_3,\\[1mm]
 [e_1, e_4]&=&q e_4, &
[e_2, e_4]&=&e_3
\end{array}$ & $
\begin{array}{l}
(1+q)  e_1^*\wedge e_3^*+ e_2^*\wedge e_4^*, \\[2mm]
\text{where } q\not=-1
\end{array}
$
& $\bar 1$ \\[2mm] \hline
 % $D^{10}_{0}$  & $
%\begin{array}{lcrlcl}
%[e_1, e_2]&=&e_2,\\[1mm]
% [e_1, e_3]&=&e_3, \\[1mm]
%[e_1, e_4]&=&0, \\[1mm]
%[e_2, e_4]&=&e_3
%\end{array}$ & $ \lambda e_1^*\wedge e_3^*+\mu e_1^*\wedge e_4^*+\lambda e_2^*\wedge e_4^*$, where $\lambda \not =0$ & $\od$ \\[2mm] \hline
  %$D^{10}_{-2}$ &
%$
%\begin{array}{lcrlcl}
%[e_1, e_2]&=&e_2,\\[1mm]
%[e_1, e_3]&=&-e_3,\\[1mm]
% [e_1, e_4]&=&-2 e_4, \\[1mm]
% [e_2, e_4]&=&e_3
%\end{array}$ & $\begin{array}{l}
%\lambda (e_2^*\wedge e_4^*- e_1^*\wedge %e_3^*)+\mu e_1^*\wedge e_4^*+\nu e_2^*\wedge %e_3^*,\\[2mm]
%\text{where  }\mu \nu+\lambda ^2\not =0
%\end{array}
%$& $\bar 1$ \\[2mm] \hline
\end{tabular}
\caption{Trivial algebras (i.e. $[\mathfrak{g}_{\bar 1}, \mathfrak{g}_{\bar 1}]=\{0\}$) with $\mathrm{sdim}=2|2$} \label{tab1}
\end{table}

\footnotesize
\begin{table}[H]
\centering
\small
\begin{tabular}{| c | l | c |c|}\hline
 The LSA & Relations in the basis: $e_1, e_2 \, |\, e_3, e_4$ & Symplectic structure & $|\omega|$ \\ \hline
 $(D^7_{1/2\;1/2})^1$ & $
\begin{array}{lcrlcl}
[e_1, e_2]&=&e_2, & [e_1, e_3]&=&\frac{1}{2} e_3\\[1mm]
 [e_1, e_4]&=&\frac{1}{2}e_4, & [e_3,e_3]&=&e_2  \\[1mm]
[e_4, e_4]&=&e_2
\end{array}$ & $e_1^*\wedge e_2^* - \frac{1}{2} e_3^*\wedge e_3^*- \frac{1}{2} e_4^*\wedge e_4^*$ & $ \bar 0$\\ \hline
 $(D^7_{1/2\;1/2})^2$ & $
\begin{array}{lcrlcl}
[e_1, e_2]&=&e_2, & [e_1, e_3]&=&\frac{1}{2} e_3\\[1mm]
 [e_1, e_4]&=&\frac{1}{2}e_4, & [e_3,e_3]&=&e_2  \\[1mm]
[e_4, e_4]&=& - e_2
\end{array}$ &  $-e_1^*\wedge e_2^* + \frac{1}{2} e_3^*\wedge e_3^* -  \frac{1}{2} e_4^*\wedge e_4^*$ & $\bar 0$\\ \hline 
 $(D^7_{1/2\;1/2})^3$ & $
\begin{array}{lcrlcl}
[e_1, e_2]&=&e_2, & [e_1, e_3]&=&\frac{1}{2} e_3\\[1mm]
 [e_1, e_4]&=&\frac{1}{2}e_4, & [e_3,e_3]&=&e_2 
\end{array}$ & None &$-$\\ \hline 
 $\begin{array}{c}
(D^7_{1-p,p})\\[1mm]
p\leq \frac{1}{2}
\end{array}$ & $
\begin{array}{lcllcl}
[e_1, e_2]&=&e_2, &
 [e_1, e_3]&=& p e_3\\[1mm]
 [e_1, e_4]&=&(1-p) e_4, & 
 [e_3,e_4]&=&e_2 
\end{array}$ & $e_1^*\wedge e_2^* -  e_3^*\wedge e_4^*$ & $\bar 0$\\ \hline 
 $(D^8_{1/2})$ & $
\begin{array}{lcllcl}
[e_1, e_2]&=&e_2, &[e_1, e_3]&=& \frac{1}{2} e_3\\[1mm]
 [e_1, e_4]&=&e_3+\frac{1}{2} e_4, & 
 [e_4,e_4]&=&e_2 
\end{array}$ & None & $-$\\ \hline 
 $
\begin{array}{c}
(D^9_{1/2,p})\\[1mm]
p>0
\end{array}$ & $
\begin{array}{lcllcl}
[e_1, e_2]&=&e_2,  & [e_3,e_3]&=&e_2 \\[1mm]
 [e_1, e_4]&=&p e_3+\frac{1}{2} e_4, & [e_4,e_4]&=&e_2 \\[1mm]
[e_1, e_3]&=& \frac{1}{2} e_3 -pe_4
\end{array}$ & $e_1^*\wedge e_2^* - \frac{1}{2}e_3^* \wedge e_3^* - \frac{1}{2}e_4^*\wedge e_4^*$ & $\ev$\\ \hline
 $(D^{10}_{0})^1$ & $
\begin{array}{lcllcl}
[e_1, e_2]&=&e_2,  & [e_1,e_3]&=&e_3 \\[1mm]
 [e_2, e_4]&=& e_3, & [e_4,e_4]&=&e_1 \\[1mm]
[e_3, e_4]&=&- \frac{1}{2} e_2
\end{array}$ & $\begin{array}{l}
2 e_2^*\wedge e_1^*-e_3^*\wedge e_4^* \\[1mm]
e_1^*\wedge e_3^*+e_2^*\wedge e_4^*
\end{array}$ & $\begin{array}{l} {\bar 0}\\ [1mm]{\bar 1}\end{array}$ \\ \hline
 $(D^{10}_{0})^2$ & $
\begin{array}{lcllcl}
[e_1, e_2]&=&e_2,  & [e_1,e_3]&=&e_3 \\[1mm]
 [e_2, e_4]&=& e_3, & [e_4,e_4]&=&-e_1 \\[1mm]
[e_3, e_4]&=& \frac{1}{2} e_2
\end{array}$ & $\begin{array}{l}
2 e_2^*\wedge e_1^*-e_3^*\wedge e_4^* \\[1mm]
e_1^*\wedge e_3^*+e_2^*\wedge e_4^*
\end{array}$ & $\begin{array}{l} {\bar 0}\\ [1mm]{\bar 1}\end{array}$ \\ \hline
 $(2A_{1,1}+2A)^1$ & $
\begin{array}{lcllcl}
[e_3, e_3]&=&e_1,  & [e_4,e_4]&=&e_2 
\end{array}$ & None & $-$\\ \hline
 $(2A_{1,1}+2A)^2$ &   $
\begin{array}{lcllcl}
[e_3, e_3]&=&e_1,  & [e_4,e_4]&=&e_2 \\[1mm]
[e_3, e_4]&=&e_1\\
\end{array}$ & None & $-$\\ \hline
$\begin{array}{c} (2A_{1,1}+2A)^3_p\\ p>0\end{array}$ & $
\begin{array}{lcllcl}
[e_3, e_3]&=&e_1,  & [e_4,e_4]&=&e_2 \\[1mm]
[e_3, e_4]&=&p(e_1+e_2)\\
\end{array}$  & $\begin{array}{l}
\text{For $p\not =\frac{1}{2}$: None}\\[1mm]
\text{For $p=\frac{1}{2}:\, e_2^*\wedge e_3^* - e_1^*\wedge e_4^*$}
\end{array}$ &
$\begin{array}{l}  
-\\
\bar 1
\end{array}
$\\ \hline
 $\begin{array}{c}(2A_{1,1}+2A)^4_p \\p>0\end{array}$ & $
\begin{array}{lcllcl}
[e_3, e_3]&=&e_1,  & [e_4,e_4]&=&e_2 \\[1mm]
[e_3, e_4]&=&p(e_1-e_2)\\
\end{array}$ & None &
$-$\\ \hline
 $(C^1_1+A)$ & $
\begin{array}{lcllcl}
[e_1, e_2]&=&e_2,  & [e_1,e_3]&=&e_3 \\[1mm]
[e_3, e_4]&=&e_2\\
\end{array}$ & $\frac{1}{2} e_2^*\wedge e_1^* +\frac{1}{2} e_3^*\wedge e_4^* +\frac{1}{4} e_4^*\wedge e_4^*$ & $\bar 0$\\ \hline
 $(C^1_{1/2}+A)$ & $
\begin{array}{lcllcl}
[e_1, e_2]&=&e_2,  & [e_1,e_3]&=&\frac{1}{2}e_3 \\[1mm]
[e_3, e_3]&=&e_2\\
\end{array}$ & $\begin{array}{l}
\lambda e_1^*\wedge e_2^* - \frac{\lambda}{2} e_3^*\wedge e_3^* - \frac{\mu}{2}e_4^*\wedge e_4^*, \\
\text{where }\lambda\mu\not =0
\end{array}$ & $\bar 0$\\ \hline
 $(C^2_{-1}+A)$ & $
\begin{array}{lcllcl}
[e_1, e_3]&=&e_3,  & [e_1,e_4]&=& - e_4 \\[1mm]
[e_3, e_4]&=&e_2\\
\end{array}$ & None & $-$\\ \hline
 $(C^3+A)$ & $
\begin{array}{lcllcl}
[e_1, e_4]&=&e_3,  & [e_4,e_4]&=& e_2 
\end{array}$ & $2 e_1^*\wedge e_2^* - e_3^*\wedge e_4^*$ & $\bar 0$\\ \hline
 $(C^5_0+A)$ & $
\begin{array}{lcllcl}
[e_1, e_3]&=& - e_4,  & [e_1,e_4]&=& e_3 \\[1mm]
[e_3, e_3]&=&e_2, & [e_4, e_4]&=&e_2\\
\end{array}$ & None & $-$\\
\hline
\end{tabular}
\caption{Non-trivial algebras (i.e. $[\mathfrak{g}_{\bar 1}, \mathfrak{g}_{\bar 1}]\not=\{0\}$) with $\mathrm{sdim}=2|2$}\label{tab2}
\end{table}

\scriptsize
\begin{table}[H]
\centering
\begin{tabular}{| c | l | c |c|}\hline
 The LSA & Relations in the basis: $e_1, e_2, e_3 \, |\, e_4$ & Symplectic structure & $|\omega|$ \\ \hline
 $D^1$ & $
\begin{array}{lcrlcl}
[e_2, e_3]&=&e_1, & [e_2, e_4]&=&e_4
\end{array}$ & $\begin{array}{l}
\lambda e_1^*\wedge e_3^*+ \mu e_2\wedge e_4^*\\[1mm]
+\nu e_1^*\wedge e_2^*+ \gamma e_2\wedge e_3^*\\[1mm]
\text{where } \lambda \mu \not =0
\end{array}$ & NH\\ \hline
 $\begin{array}{l}
D^2_q\\[1mm]
q \not =-1,0
\end{array}$ & $
\begin{array}{lcrlcl}
[e_1, e_3]&=& e_1, & [e_2, e_3]&=& e_1+ e_2,\\[1mm]
 [e_3, e_4]&=& q e_4
\end{array}$ & None & $-$\\ \hline
 $\begin{array}{l}
D^2_{-1}
\end{array}$ & $
\begin{array}{lcrlcl}
[e_1, e_2]&=& e_1, & [e_2, e_3]&=& e_1+ e_2,\\[1mm]
 [e_3, e_4]&=& - e_4
\end{array}$ & $
\begin{array}{l}
\lambda e_1^*\wedge e_3^*+ \mu e_2^*\wedge e_3^*+\\[2mm]
 \nu e_2^*\wedge e_4^*+ \gamma e_3^*\wedge e_4^*\\[2mm]
\text{where } \lambda \nu \not =0
\end{array}$ & NH\\ \hline
 $\begin{array}{l}
D^3_{pq}\\[1mm]
pq\not =0
\end{array}$ & $
\begin{array}{lcllcl}
[e_1, e_3]&=& pe_1-e_2, & [e_2, e_3]&=& e_1+ pe_2,\\[1mm]
 [e_3, e_4]&=&q e_4
\end{array}$ & None & $-$\\ \hline
\end{tabular}
\caption{Trivial algebras (i.e. $[\mathfrak{g}_{\bar 1}, \mathfrak{g}_{\bar 1}]=\{0\}$) with $\mathrm{sdim}=3|1$}\label{tab3}
\end{table}

%\nopagebreak

\hspace{-9cm}

\scriptsize
\begin{table}[H]
\centering
\begin{tabular}{| c | l | c |c|}\hline
 The LSA & Relations  in the basis: $e_1\,|\,e_2,e_3,e_4$ & Symplectic structure & $|\omega|$ \\ \hline
 $\begin{array}{c}
D^{11}_{pq}\\
0< |p|\leq |q|\leq 1
\end{array}$ & $
\begin{array}{lcrlcl}
[e_1, e_2]&=&e_2, & [e_1, e_3]&=&p e_3,\\[1mm]
 [e_1, e_4]&=&q e_4
\end{array}$ & None& $-$\\ \hline
$\begin{array}{c}
D^{11}_{pq}\\
(p,q)=(-1,-1)
\end{array}$ & $
\begin{array}{lcrlcl}
[e_1, e_2]&=&e_2, & [e_1, e_3]&=&p e_3,\\[1mm]
 [e_1, e_4]&=&q e_4
\end{array}$ & $
\begin{array}{l}\lambda e_1^*\wedge e_2^*+ \mu e_1^*\wedge e_3^*+\nu e_1^*\wedge e_4^*\\[1mm]
+\gamma e_2^*\wedge e_3^*+\delta e_2^*\wedge e_4^*,\\[1mm]
\text{where } \nu \gamma-\mu \delta\not =0
\end{array}$ & NH\\ \hline
$\begin{array}{c}
D^{11}_{pq}\\
0< |p|\leq 1\\
p=-q
\end{array}$ & $
\begin{array}{lcrlcl}
[e_1, e_2]&=&e_2, & [e_1, e_3]&=&p e_3,\\[1mm]
 [e_1, e_4]&=&q e_4
\end{array}$ & $
\begin{array}{l}\lambda e_1^*\wedge e_2^*+ \mu e_1^*\wedge e_3^*+\nu e_1^*\wedge e_4^*\\[1mm]
+\gamma e_3^*\wedge e_4^*,\text{ where } \lambda \gamma\not =0
\end{array}$ & NH\\ \hline
$\begin{array}{c}
D^{11}_{pq}\\
0< |p|< 1\\
q=-1
\end{array}$ & $
\begin{array}{lcrlcl}
[e_1, e_2]&=&e_2, & [e_1, e_3]&=&p e_3,\\[1mm]
 [e_1, e_4]&=&q e_4
\end{array}$ & $
\begin{array}{l}\lambda e_1^*\wedge e_2^*+ \mu e_1^*\wedge e_3^*+\nu e_1^*\wedge e_4^*\\[1mm]
+\gamma e_2^*\wedge e_4^*, \text{ where } \mu \gamma\not =0
\end{array}$
 & NH\\ \hline
 $D^{12}$ & $
\begin{array}{lcrlcl}
[e_1, e_2]&=&e_2, & [e_1, e_4]&=&e_3,
\end{array}$ & None & $-$\\ \hline
 $\begin{array}{c}
D^{13}_p\\[1mm]
p\not =0 \text{ (generic)}
\end{array}$ & $
\begin{array}{lcllcl}
[e_1, e_2]&=&p e_2, & [e_1, e_3]&=&e_3,\\[1mm]
 [e_1, e_4]&=&e_3+e_4
\end{array}$ & None & $-$\\ \hline
$\begin{array}{c}
D^{13}_{-1}
\end{array}$ & $
\begin{array}{lcllcl}
[e_1, e_2]&=& - e_2, & [e_1, e_3]&=&e_3,\\[1mm]
 [e_1, e_4]&=&e_3+e_4
\end{array}$ & $
\begin{array}{l}
\lambda e_1^*\wedge e_2^*+ \mu e_1^*\wedge e_3^*+ \gamma e_1^*\wedge e_4^*\\[1mm]
+\nu e_2^*\wedge e_4^*, \text{ where } \mu \nu\not =0
\end{array}$ & NH\\ \hline
$\begin{array}{c}
D^{14}_{pq}\\[1mm]
p\not =0, q\geq 0
\end{array}$ & $
\begin{array}{lcllcl}
[e_1, e_2]&=&p e_2, \\[1mm]
 [e_1, e_3]&=&q e_3-e_4,\\[1mm]
 [e_1, e_4]&=&q e_4+e_3
\end{array}$ & None & $-$\\ \hline
$D^{15}$ & $
\begin{array}{lcrlcl}
[e_1, e_3]&=&e_2 ,& [e_1, e_4]&=&e_3,
\end{array}$ & 
$\begin{array}{l}
\lambda e_1^*\wedge e_2^*+\mu e_1^*\wedge e_3^*+ \nu e_1^*\wedge e_4^* \\[1mm]
+ \delta e_2^*\wedge e_4^* -\frac{1}{2} \delta e_3^*\wedge e_3^*,\\[1mm]
\text{where } \delta^2(\mu^2-2 \lambda  \nu)\not =0.
\end{array}$ & NH\\ \hline
$D^{16}$ & $
\begin{array}{lcl}
[e_1, e_2]&=&e_2 ,\\[1mm]
 [e_1, e_3]&=&e_2+e_3,\\[1mm]
 [e_1, e_4]&=&e_3+e_4
\end{array}$ & None & $-$\\ \hline
\end{tabular}
\caption{Trivial algebras (i.e. $[\mathfrak{g}_{\bar 1}, \mathfrak{g}_{\bar 1}]=\{0\}$) with $\mathrm{sdim}=1|3$}\label{tab4}
\end{table}

\footnotesize

\begin{table}[H]
\centering
\begin{tabular}{| c | l | c |c|}\hline
 The LSA & Relations in the basis: $e_1, e_2, e_3 \, |\, e_4$& Symplectic structure & $|\omega|$ \\ \hline
 $(A_{3,1}+A)$ & $
\begin{array}{lcrlcl}
[e_2, e_3]&=&e_1, & [e_4, e_4]&=&e_1
\end{array}$ & None & $-$\\ \hline
$\begin{array}{l}
(D^3_{p, -1/2})\\[1mm]
p\not =0
\end{array}$ & $
\begin{array}{lcrlcl}
[e_1, e_2]&=&e_2, & [e_1, e_3]&=&p e_3 , \\[1mm]
[e_1, e_4]&=&\frac{1}{2}e_4, & [e_4, e_4]&=& e_2\\[1mm]
\end{array}$ & None & $-$\\ \hline
$(D^3_{-1/2, -1/2})$ & $
\begin{array}{lcrlcl}
[e_1, e_2]&=&e_2, & [e_1, e_3]&=&- \frac{1}{2} e_3 , \\[1mm]
[e_1, e_4]&=&\frac{1}{2}e_4, & [e_4, e_4]&=& e_2\\[1mm]
\end{array}$ & $\begin{array}{l}
\frac{1}{2} \lambda e_1^*\wedge e_2^*+ \mu e_1^*\wedge e_3^*\\[1mm]
+ \nu e_1^*\wedge e_4^*+\gamma e_3^*\wedge e_4^*\\[1mm]
+\frac{1}{2}\lambda e_4^*\wedge e_4^*, \text{ where }\gamma \lambda\not =0
\end{array}$ & NH\\ \hline
 $(D^2_{-1/2})^1$ & $
\begin{array}{lcrlcl}
[e_1, e_2]&=&e_2, & [e_1, e_3]&=&e_2+e_3 , \\[1mm]
[e_1, e_4]&=&\frac{1}{2}e_4, & [e_4, e_4]&=&e_2
\end{array}$ & None & $-$\\ \hline
 $(D^2_{-1/2})^2$ & $
\begin{array}{lcrlcl}
[e_1, e_2]&=&e_2, & [e_1, e_3]&=&-e_2+e_3 , \\[1mm]
[e_1, e_4]&=&\frac{1}{2}e_4, & [e_4, e_4]&=&e_2
\end{array}$ & None & $-$\\ \hline
\end{tabular}
\caption{Trivial algebras (i.e. $[\mathfrak{g}_{\bar 1}, \mathfrak{g}_{\bar 1}]\not =\{0\}$) with $\mathrm{sdim}=3|1$}\label{tab5}
\end{table}

\footnotesize
\begin{table}[H]
\centering
\begin{tabular}{| c | l | c |c|}\hline
 The LSA & Relations in the basis: $e_1\,|\,e_2,e_3,e_4$ & Symplectic structure & $|\omega|$ \\ \hline
 $(A_{1,1}+3 A)^1$ & $
\begin{array}{lcrlcl}
[e_2, e_2]&=&e_1, & [e_3, e_3]&=&e_1,\\[1mm]
 [e_4, e_4]&=&e_1
\end{array}$ & None & $-$\\ \hline
$(A_{1,1}+3 A)^2$ & $
\begin{array}{lcrlcl}
[e_2, e_2]&=&e_1, & [e_3, e_3]&=&e_1,\\[1mm]
 [e_4, e_4]&=&-e_1
\end{array}$ & None & $-$\\ \hline
\end{tabular}
\caption{Trivial algebras (i.e. $[\mathfrak{g}_{\bar 1}, \mathfrak{h}_{\bar 1}]\not =\{0\}$) with $\mathrm{sdim}=1|3$}\label{tab6}
\end{table}

\normalsize 

%%%%%%%%%%%%%%%%%%%%%%%%%%%%%%%%%%%%%%%%%%%%%%%%%
\section{Appendix: The decomposable case}
%%%%%%%%%%%%%%%%%%%%%%%%%%%%%%%%%%%%%%%%%%%%%%%%%%
We are only considering Lie superalgebras that are {\bf not} Lie algebras. We investigate symplectic structures on the decomposable 4-dimensional real Lie superalgebras below. The notation we use is that of Backhouse \cite{B}. Over $A\oplus B$, the notation $\omega:=\omega_A \oplus \omega_B$ means that
$\omega(a_1+b_1,a_2+b_2)=\omega_A(a_1,a_2)+\omega_B(b_1,b_2)$ for all $a_1,a_2\in A$ and $b_1,b_2\in B$.

\subsection{Abelian LSA}

\begin{enumerate}
\item NH: $\mathbb{R}^{1|3},$ $\mathbb{R}^{3|1},$
\item Symplectic: $(\mathbb{R}^{2|2}, \omega_{\mathbb{R}^{1|1}} \oplus \omega_{\mathbb{R}^{1|1}})$,  $(\mathbb{R}^{2|2}, \omega_{\mathbb{R}^{2|0}} \oplus \omega_{\mathbb{R}^{0|2}})$, $(\mathbb{R}^{0|4},\omega_{\mathbb{R}^{0|2}} \oplus \omega_{\mathbb{R}^{0|2}}).$

\end{enumerate}

\subsection{LSA of the form \texorpdfstring{$L\oplus M$}{L+M}, where \texorpdfstring{$L, M$}{L,M} are 2-dimensional} We refer to Table \ref{tab7} for the description of the LSA below:

\begin{enumerate}
\item NH: $B\oplus \mathbb{R}^{2|0},$ $B\oplus \mathbb{R}^{0|2},$ $B\oplus \mathrm{Al},$ $ \mathrm{Al}\oplus \mathbb{R}^{1|1},$ 
\item None: $(A_{1,2}+A)\oplus \mathbb{R}^{2|0}$, $(A_{1,2}+A)\oplus \mathbb{R}^{0|2},$ $(A_{1,2}+A)\oplus \mathbb{R}^{1|1}$, $(A_{1,2}+A)\oplus \mathrm{Al}$, $(A_{1,2}+A)\oplus B$,  $(A_{1,2}+A)\oplus (A_{1,2}+A)$,
\item Symplectic: $(B\oplus B, \omega_B\oplus \omega_B)$, $(B\oplus \mathbb{R}^{1|1}, \omega_B\oplus \omega_{\mathbb{R}^{1|1}})$, $(\mathrm{Al}\oplus \mathbb{R}^{0|2}, \omega_{\mathrm{Al}}\oplus \omega_{\mathbb{R}^{0|2}})$.
\end{enumerate}

\footnotesize
\begin{table}[H]
\centering
\begin{tabular}{| c | l | c |c|}
\hline
 The LA & Relations in the basis: $e_1,e_2$ & Symplectic structure & $|\omega|$ \\ \hline
 $\mathrm{Al}$ & $
\begin{array}{lcr}
[e_1, e_2]&=&e_1 
\end{array}$ & $\omega_{\mathrm{Al}}:=e_1^*\wedge e_2^*$ & $ \overline{0}$\\ \hline
 $\mathbb{R}^{2|0}$ &  & $\omega_{\mathbb{R}^{2|0}}:=e_1^*\wedge e_2^*$ & $ \overline{0}$\\ \hline
 The LSA & Relations in the basis: $e_1\,|\,e_2$ & Symplectic structure & $|\omega|$ \\ \hline
 $B$ & $
\begin{array}{lcr}
[e_1, e_2]&=&e_2 
\end{array}$ & $\omega_B:=e_1^*\wedge e_2^*$ & $ \overline{1}$\\ \hline
$(A_{1,2}+A)$ & $
\begin{array}{lcr}
[e_2, e_2]&=&e_1
\end{array}$ & None & $-$\\ \hline
$\mathbb{R}^{1|1}$ &  & $\omega_{\mathbb{R}^{1|1}}:=e_1^*\wedge e_2^*$ & $ \overline{1}$\\ \hline
The LSA & Relations in the basis: $|\,e_1, e_2$ & Symplectic structure & $|\omega|$ \\ \hline
$\mathbb{R}^{0|2}$ &  & $\omega_{\mathbb{R}^{0|2}}:=e_1^*\wedge e_2^*$ & $ \overline{0}$\\ \hline
\end{tabular}
\caption{Quasi-Frobenuis structure}\label{tab7}
\end{table}

\normalsize 

\subsection{LSA of the form \texorpdfstring{$L\oplus \mathbb{R}^{0|1}$}{L+R} or \texorpdfstring{$L\oplus \mathbb{R}^{1|0}$,}{L+R}  where \texorpdfstring{$L$}{L} is 3-dimensional}
Below, $X$ represents the generator of either $\mathbb{R}^{1|0}$ or $\mathbb{R}^{0|1}$. The parity of $X$ should be understood from the context.

\footnotesize
\begin{table}[H]
\centering
\begin{tabular}{| c | l | c |c|}\hline
 The LSA & Relations in the basis: $e_1,e_2\,|\,e_3$ and $X$ & Symplectic structure & $|\omega|$ \\ \hline
$C^1_p\oplus \mathbb{R}^{1|0}, p\not =0$ & $
\begin{array}{lcrlcr}
[e_1,e_2]&=&e_2, & [e_1,e_3]&=&pe_3
\end{array}$ & None & $-$\\ \hline
$C^1_p\oplus \mathbb{R}^{0|1}, p\not =0$ & $
\begin{array}{lcrlcr}
[e_1,e_2]&=&e_2, & [e_1,e_3]&=&pe_3
\end{array}$ & None & $-$\\ \hline
$C^1_{-1}\oplus \mathbb{R}^{1|0}$ & $
\begin{array}{lcrlcr}
[e_1,e_2]&=&e_2, & [e_1,e_3]&=&-e_3
\end{array}$ & NH & $-$\\ \hline
$C^1_{-1}\oplus \mathbb{R}^{0|1}$ & $
\begin{array}{lcrlcr}
[e_1,e_2]&=&e_2, & [e_1,e_3]&=&-e_3
\end{array}$ & $e_1^*\wedge X^*+e_2^*\wedge e_3^*$ & $\bar 1$\\ \hline
$C^1_{1/2}\oplus \mathbb{R}^{1|0}$ & $
\begin{array}{lcrlcr}
[e_1,e_2]&=&e_2, & [e_1,e_3]&=& \frac{1}{2}e_3,\\[1mm]
[e_3,e_3]&=&e_2
\end{array}$ & None & $-$\\ \hline
$C^1_{1/2}\oplus \mathbb{R}^{0|1}$ & $
\begin{array}{lcrlcr}
[e_1,e_2]&=&e_2, & [e_1,e_3]&=& \frac{1}{2}e_3,\\[1mm]
[e_3,e_3]&=&e_2
\end{array}$ & $e_3^*\wedge e_3^*+X^*\wedge X^*$ & $\bar 0$\\ \hline
 The LSA & Relations in the basis: $e_1\,|\,e_3, e_4$ and $X$ & Symplectic structure & $|\omega|$ \\ \hline
$C^2_p\oplus \mathbb{R}^{1|0},\; 0<|p|\leq 1$ & $
\begin{array}{lcrlcr}
[e_1, e_3]&=&e_3, & [e_1, e_4]&=&pe_4
\end{array}$ & None & $-$\\ \hline
$C^2_p\oplus \mathbb{R}^{0|1}, \; 0<|p|\leq 1$ & $
\begin{array}{lcrlcr}
[e_1, e_3]&=&e_3, & [e_1, e_4]&=&pe_4
\end{array}$ & None & $-$\\ \hline
$C^2_{-1}\oplus \mathbb{R}^{1|0}$ & $
\begin{array}{lcrlcr}
[e_1, e_3]&=&e_3, & [e_1, e_4]&=&-e_4
\end{array}$ & $e_1^*\wedge e_4^*+X^*\wedge e_4^*$ & $\bar 1$\\ \hline
$C^2_{-1}\oplus \mathbb{R}^{0|1}$ & $
\begin{array}{lcrlcr}
[e_1, e_3]&=&e_3, & [e_1, e_4]&=&-e_4
\end{array}$ & NH & $-$\\ \hline
$C^3\oplus \mathbb{R}^{1|0}$ & $
\begin{array}{lcr}
[e_1,e_4]&=&e_3
\end{array}$ & $e_1^*\wedge e_3^*+X^*\wedge e_4^*$ & $\bar 1$\\ \hline
$C^3\oplus \mathbb{R}^{0|1}$ & $
\begin{array}{lcr}
[e_1,e_4]&=&e_3
\end{array}$ & NH & $-$\\ \hline
$C^4\oplus \mathbb{R}^{1|0}$ & $
\begin{array}{lcrlcr}
[e_1, e_3]&=&e_3,& [e_1, e_4]=e_3+e_4
\end{array}$ & None & $-$\\ \hline
$C^4\oplus \mathbb{R}^{0|1}$ & $
\begin{array}{lcrlcr}
[e_1, e_3]&=&e_3,& [e_1, e_4]=e_3+e_4
\end{array}$ & None & $-$\\ \hline
$C^5_p\oplus \mathbb{R}^{1|0}, \; p\geq 0$ & $
\begin{array}{lcrlcr}
[e_1, e_3]&=&p e_3-e_4, & [e_1, e_4]&=& e_3+p e_4
\end{array}$ & None & $-$\\ \hline
$C^5_p\oplus \mathbb{R}^{0|1}, \; p\geq 0$ & $
\begin{array}{lcrlcr}
[e_1, e_3]&=&p e_3-e_4, & [e_1, e_4]&=& e_3+p e_4
\end{array}$ & None & $-$\\ \hline 
$(A_{1,1}+2A)^1\oplus \mathbb{R}^{1|0}$ & $
\begin{array}{lcrlcr}
[e_3, e_3]&=&e_1, & [e_4, e_4]=e_1
\end{array}$ & None & $-$\\ \hline
$(A_{1,1}+2A)^1\oplus \mathbb{R}^{0|1}$ & $
\begin{array}{lcrlcr}
[e_3, e_3]&=&e_1, & [e_4, e_4]=e_1
\end{array}$ & None & $-$\\ \hline
$(A_{1,1}+2A)^2\oplus \mathbb{R}^{1|0}$ & $
\begin{array}{lcrlcr}
[e_3, e_3]&=&e_1, & [e_4, e_4]=-e_1
\end{array}$ & None & $-$\\ \hline
$(A_{1,1}+2A)^2\oplus \mathbb{R}^{0|1}$ & $
\begin{array}{lcrlcr}
[e_3, e_3]&=&e_1, & [e_4, e_4]=-e_1
\end{array}$ & None & $-$\\ \hline
\end{tabular}
\caption{Quasi-Frobenuis structure}\label{tab7last}
\end{table}

\normalsize 

% \section{A short summary}
% The main goal of this paper has been to provide a comprehensive analysis 
% of the 4-dimensional Lie superalgebras classified by Backhouse~\cite{B}. 
% The study culminates in the following two principal results:

% (i) First, we determine which Lie superalgebras arise as Lagrangian extensions. It turns out that, among the Lie superalgebras under consideration, 
% 22 do not admit any quasi-Frobenius structure, whereas 9 admit a 
% nonhomogeneous quasi-Frobenius structure. 
% Among the quasi-Frobenius cases, 4 do not arise as Lagrangian extensions 
% of smaller Lie superalgebras. 
% Furthermore, 8 are realized as $T^*$-extensions and 8 as $\Pi T^*$-extensions. 
% The complete classification is summarized in Tables~\ref{tab:nonQF}--\ref{tab:PTQF}. In addition, three Lie superalgebras admit both a closed orthosymplectic and periplectic form, namely 
% $(D_0^{10})^1$, $(D_0^{10})^2$, and $D_{pq}^{7}$ with $(p,q)=(1,-1)$.

% For each Lie superalgebra arising as a $T^*$-extension or a $\Pi T^*$-extension, 
% the corresponding connection is given explicitly. 
% The method for constructing these extensions is described in Section~\ref{Lagrangian}. 

% (ii) Second, we investigate Novikov structures and Balinsky-Novikov structures for these Lie superalgebras. It turns out that all of them are Novikov except . It is interesting to classify all these strucutres on each Lie superalgebra. 

%%%%%%%%%%%%%%%%%%%%%%%%%%
\vspace{2mm}

{\bf Acknowledgment.} We would like to acknowledge Chengming Bai for pointing out that Novikov algebras admit two superizations. Also, we would like to thank Sa\"id Benayadi, Isabel Cunha, Quentin Ehret, and the referee for reading the manuscript and making pertinent comments.\\

%%%%%%%%%%%%%%%%%%%%%%%%%%%%%%
\noindent {\bf Declarations:}
\bigskip

\noindent {\bf Funding Statement}:  SB's research was supported by GRANT: AD-065. \\
{\bf Competing Interest Statement}: the authors declare no competing interest. \\
{\bf Authors' Contribution Statement}: the authors contributed equally to this work.

%\section{Conclusion}

%%%%%%%%%%%%%%%%%%%%%%%%%%%%%%%%%%%%%

\end{document}